\documentclass[12pt]{article}
\newcommand{\bfm}[1]{\mbox{\boldmath{$#1$}}}
\newcommand{\beq}{\begin{eqnarray}}
\newcommand{\eeq}{\end{eqnarray}}
\newcommand{\beqs}{\begin{eqnarray*}}
\newcommand{\eeqs}{\end{eqnarray*}}

\usepackage{graphics,epsfig}
\usepackage{latexsym}
\usepackage{amsthm}
\usepackage{amssymb,amsfonts,amsmath}

\newtheorem{theorem}{Theorem}
\newtheorem{corollary}{Corollary}

\newtheorem{definition}{Definition}
\newtheorem{lemma}{Lemma}

\begin{document}

\title{Relative Equilibria in the Spherical, Finite Density 3-Body Problem}
\author{D.J. Scheeres\\Department of Aerospace Engineering Sciences\\The University of Colorado at Boulder\\scheeres@colorado.edu}
\date{\today}
\maketitle

\begin{abstract}

The relative equilibria for the spherical, finite density 3 body problem are identified. Specifically, there are 28 distinct relative equilibria in this problem which include the classical 5 relative equilibria for the point-mass 3-body problem. None of the identified relative equilibria exist or are stable over all values of angular momentum. The stability and bifurcation pathways of these relative equilibria are mapped out as the angular momentum of the system is increased. This is done under the assumption that they have equal and constant densities and that the entire system rotates about its maximum moment of inertia. The transition to finite density greatly increases the number of relative equilibria in the 3-body problem and ensures that minimum energy configurations exist for all values of angular momentum. 

\end{abstract}

\section{Introduction}

The 3-body problem is one of the most fundamental and well studied problems in Celestial Mechanics. A key result for this problem is that there exist only 5 relative equilibria, and that these exist for all levels of angular momentum \cite{euler,lagrange}. The properties of these special solutions have been deeply studied, and have motivated significant research in mechanics and dynamics. A hallmark of the classical problem is that the bodies are considered to be point masses, with no restrictions on how close they can come to each other. A recent variation of this problem has been posed that removes this one restriction \cite{scheeres_minE}, and supposes that these bodies can be rigid bodies with finite density, and hence have limits on their proximity. Such ``Full Body'' systems inherit the fundamental symmetries of the $N$-body problem \cite{scheeres_F2BP}, however they also demand that the rotational angular momentum, energy and dynamics of these rigid bodies be incorporated in the theory as well. 

This paper studies the spherical 3-body problem under the assumption that the bodies are rigid and have finite density, and thus the separation between the bodies is constrained to be positive. This one change completely alters the character of the problem and, while the traditional Euler and Lagrange solutions still exist for large enough angular momentum values, a full 23 additional relative equilibria emerge from the analysis at all values of angular momentum, with a complex and rich bifurcation scheme. 

The celestial mechanics of bodies with finite density and fixed shape can have dynamical evolution and relative equilibria that are quite distinct from that found in the classical Newtonian point mass $N$-body problem. These differences were previously explored in \cite{scheeres_minE} where several results were proven for the so-called ``Full Body Problem,'' in which the individual bodies are treated as rigid bodies with finite densities. Specifically, it was shown that, in opposition to the point mass $N$-body problem, the full body problem will always have a minimum energy configuration. Further, the number and variety of relative equilibria for that problem are greatly enhanced, and now include configurations where the bodies can rest on each other and configurations where different collections of resting bodies orbit each other, as well as the classical central configurations. One important aspect of this problem is that the existence and stability of configurations become a function of the total angular momentum of the system, a dependance that does not exist for the classical point mass $N$-body problem. 

This paper studies the relative equilibria of one particular problem in the Full Body Problem (FBP) to completeness. Specifically, all relative equilibria of the planar spherical full 3-body problem, which consists of three spheres of equal density but arbitrary size, located in the plane perpendicular to the angular momentum. 
The explicit methodology used was developed in \cite{scheeres_minE, scheeres_minE_chap}, and is fundamentally based on analysis of the amended potential as developed by Smale \cite{smaleI, smaleII} and motivated by observations from Arnold \cite{arnold1988mathematical}. The main theorem is stated and described at first, the problem is technically defined, then several results used to make the proof are listed, and finally all the detailed computations for the proof are given. Following the proof, a summary of the proof is provided, indicating how it establishes the theorem. 

A main application of this result is to identify the stable states that can be physically achieved by a collection of self-gravitating bodies that can sustain contact. This situation happens in solar system dynamics when considering the physical nature of rubble pile asteroids \cite{fujiwara_science}, although there the number of individual grains can be quite large. Recent observations of comets, however, also show that they can be comprised of a few larger components that rest on each other and, given their changing spin rates due to outgassing, also mimic a system with changing angular momentum \cite{massironi2015two}. 
A specific motivation from this current analysis would seek out natural situations that mimic the stable members of these configurations. As the existence of some of these stable configurations are somewhat unexpected, finding such configurations in nature would be especially interesting. 

Additional avenues for exploration would be to expand the current analyses to study the dynamics of finite bodies as they interact with each other gravitationally and through impact. This sort of approach has been followed in the planetary sciences community in the study of rubble pile asteroids \cite{richardson_equilibrium, sanchez_icarus}. The current results can motivate the formation of stable shapes as a function of body morphology and total angular momentum. 
In addition, the presence of finite densities and minimum energy configurations also enables the rigorous computation of energy limits for the Hill stability in the Full $N$-body problem \cite{IAU_namur, IAU_hawaii}, an avenue of further investigation for the current problem. 

\section{Main Result}

\begin{theorem} 
\label{theorem:1}
In the spherical full 3-body problem there exist a total of 28 distinct relative equilibrium configurations. No single class of relative equilibria exists or is energetically stable for all angular momentum, however at every value of angular momentum there exists at least 1 energetically stable relative equilibria. The pattern of relative equilibria can be fully represented in a bifurcation chart as the total angular momentum of the system varies from 0 to $\infty$. 
\end{theorem}

The 28 different relative equilibria can be delineated in a few different ways. Figure \ref{fig:equilibria} shows these relative equilibria separated into 7 different classes, each with one to six distinct configurations. The figure shows 20 distinct configurations, with 8 of them having an alternate ordering not shown in the figure. 

\begin{figure}[h!]
\centering
\includegraphics[scale=0.4]{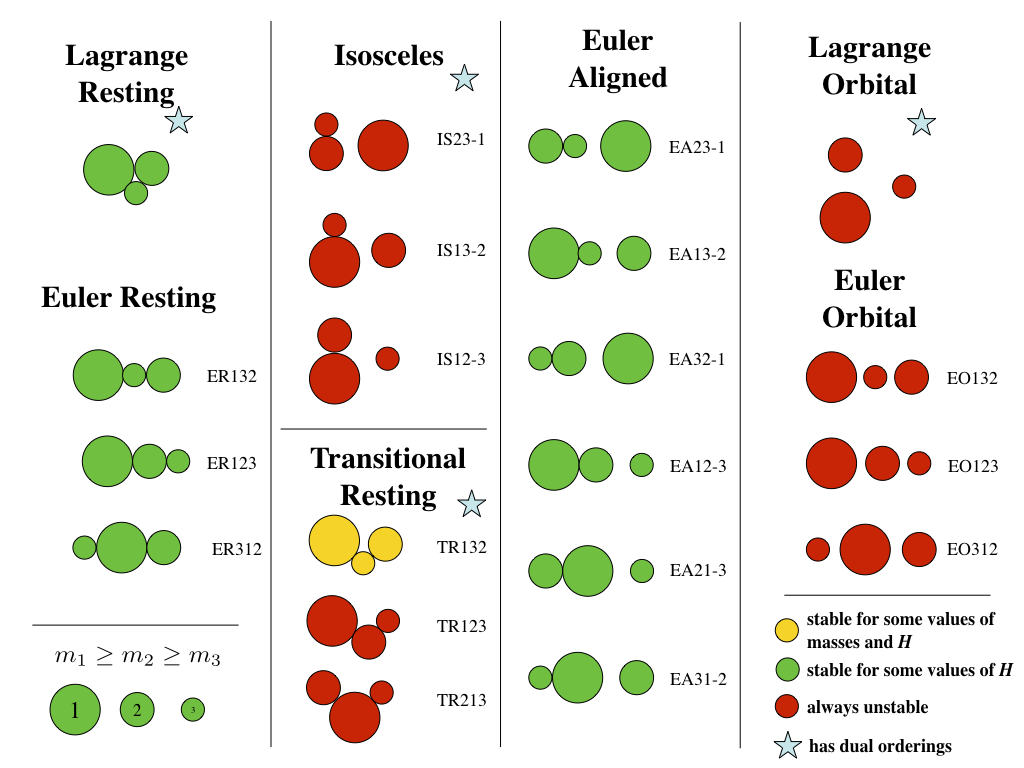}
\caption{Examples of the 20 different equilibrium configurations are shown within different classes. Colors indicate whether or not the configuration can be stable for a range of angular momentum (green), for some special combinations of mass ratio and angular momentum (yellow), or if they are unstable for all angular momentum values and mass ratios (red). Stars indicate when a reordering of the masses provides another relative equilibrium. The detailed naming convention for the configurations is also introduced. }
\label{fig:equilibria}
\end{figure}

Figures \ref{fig:bif1}-\ref{fig:bif4} show the detailed sequences of bifurcations that occur as the total angular momentum of the system is increased. These diagrams are qualitative, but the sequence of bifurcations in the specific connected pathways are accurate and will be derived in the course of the proof. Note that the sequence of bifurcations keeps some of the configurations separated from each other. In other cases the sequence of bifurcations will change as the relative masses of the bodies are changed, in these cases multiple types of sequences are shown. Note in Fig.\ \ref{fig:bif2} that this particular sequence has at least one stable configuration for all values of angular momentum while none of the others has a stable configuration in the vicinity of $H= 0$. Also note that the sequences shown in Figs.\ \ref{fig:bif1}--\ref{fig:bif3} all have 2 stable configurations for large $H$. 

\begin{figure}[h!]
\centering
\includegraphics[scale=0.35]{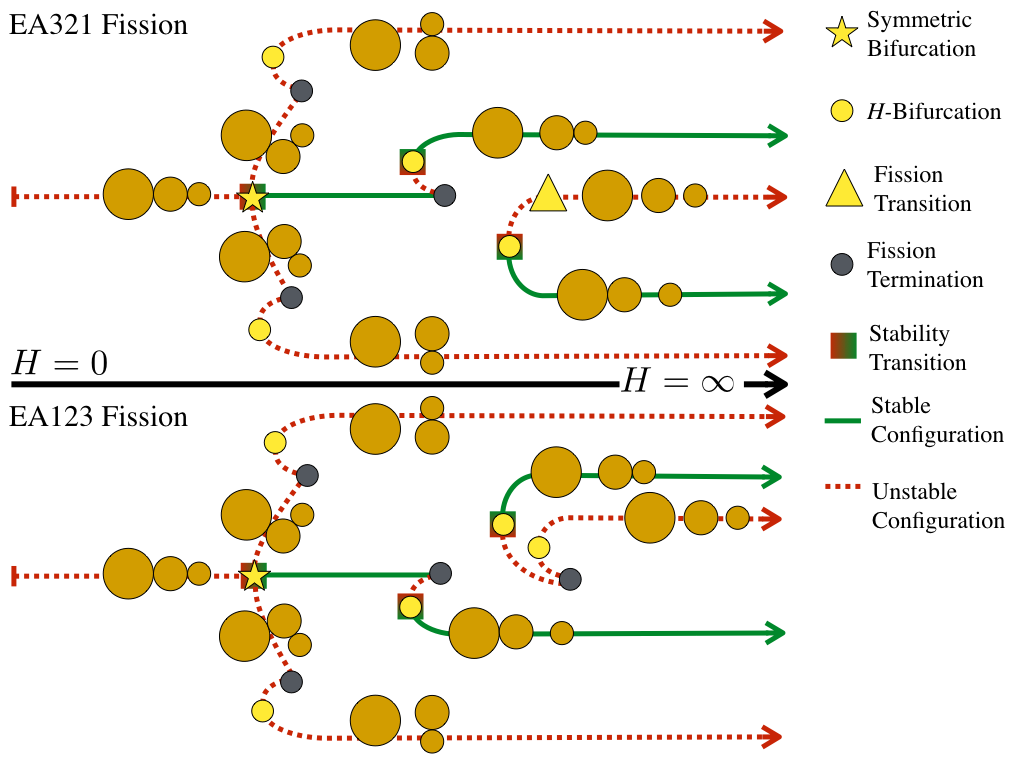}
\caption{Bifurcation diagram showing the possible branches relating to the ER123 configuration. Depending on the relative mass values two different bifurcation pathways exist. Within each pathway, the manner in which the EO configuration appears can shift between the two pathways shown. }
\label{fig:bif1}
\end{figure}

\begin{figure}[h!]
\centering
\includegraphics[scale=0.35]{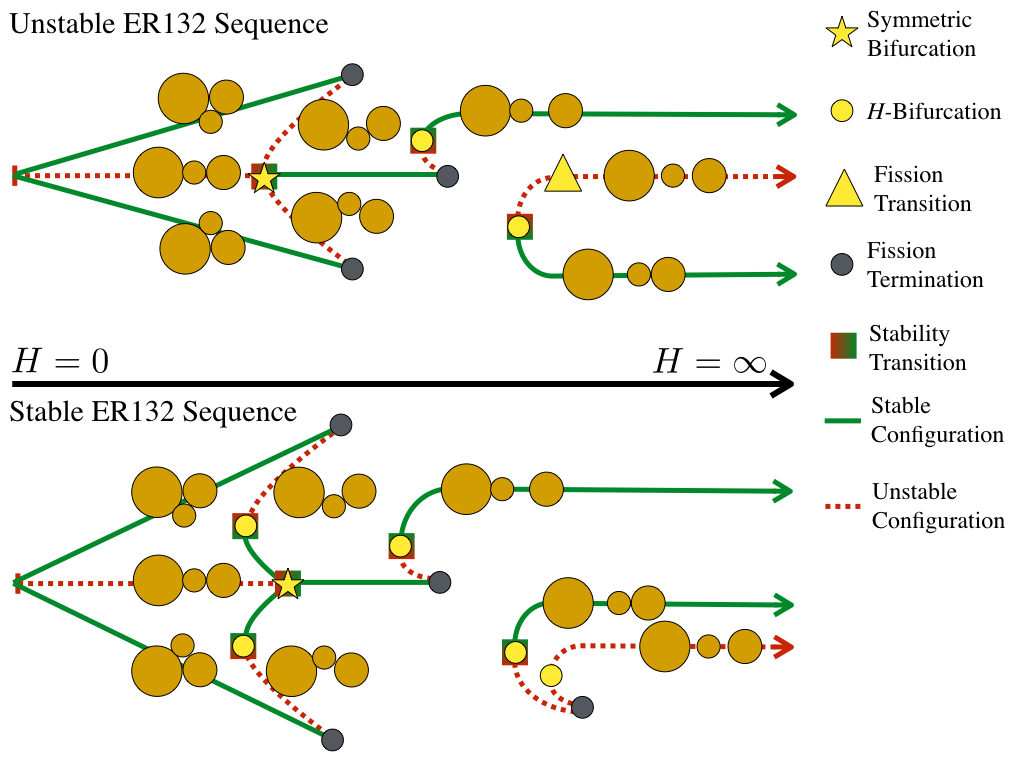}
\caption{Bifurcation diagram showing the possible branches relating to the ER132 configuration. ER132 links to the LR configuration, and always fissions into the same EA13-2 configuration. Depending on the relative mass values there can be a range of stable TR132 configurations. Within each pathway, the manner in which the EO configuration appears can also shift between the two pathways shown. The two different sequence need not have the same pattern of TR and EO bifurcations.}
\label{fig:bif2}
\end{figure}

\begin{figure}[h!]
\centering
\includegraphics[scale=0.35]{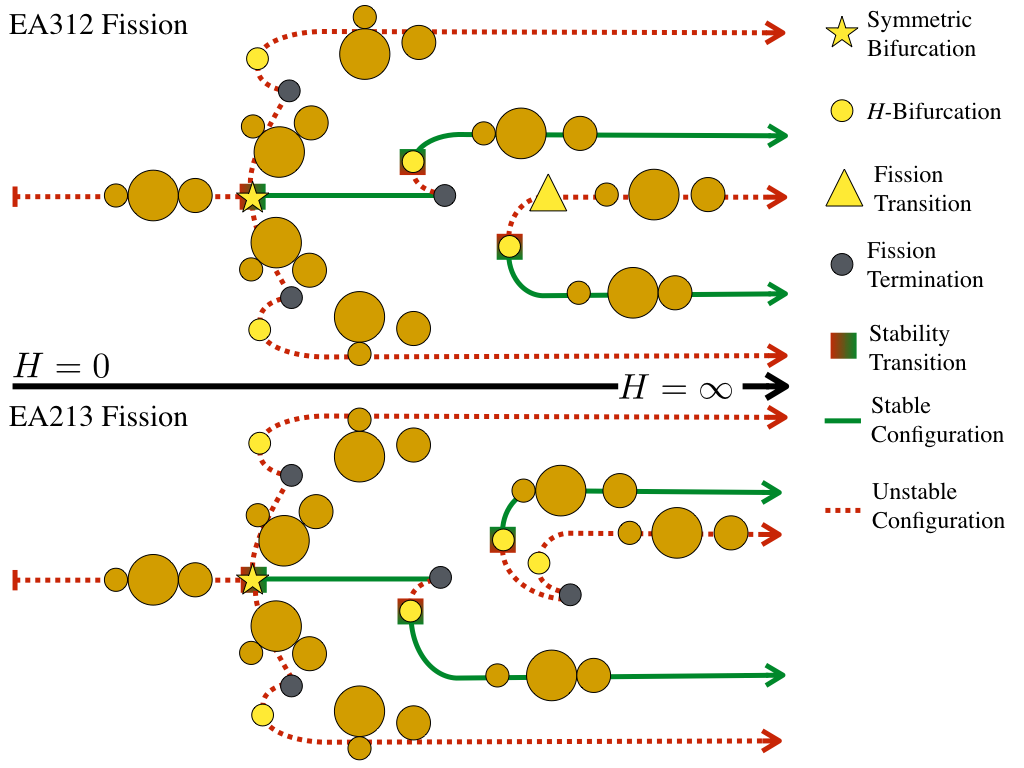}
\caption{Bifurcation diagram showing the possible branches relating to the ER213 configuration. Depending on the relative mass values two different bifurcation pathways exist. Within each pathway, the manner in which the EO configuration appears can also shift between the two pathways shown. }
\label{fig:bif3}
\end{figure}

\begin{figure}[h!]
\centering
\includegraphics[scale=0.35]{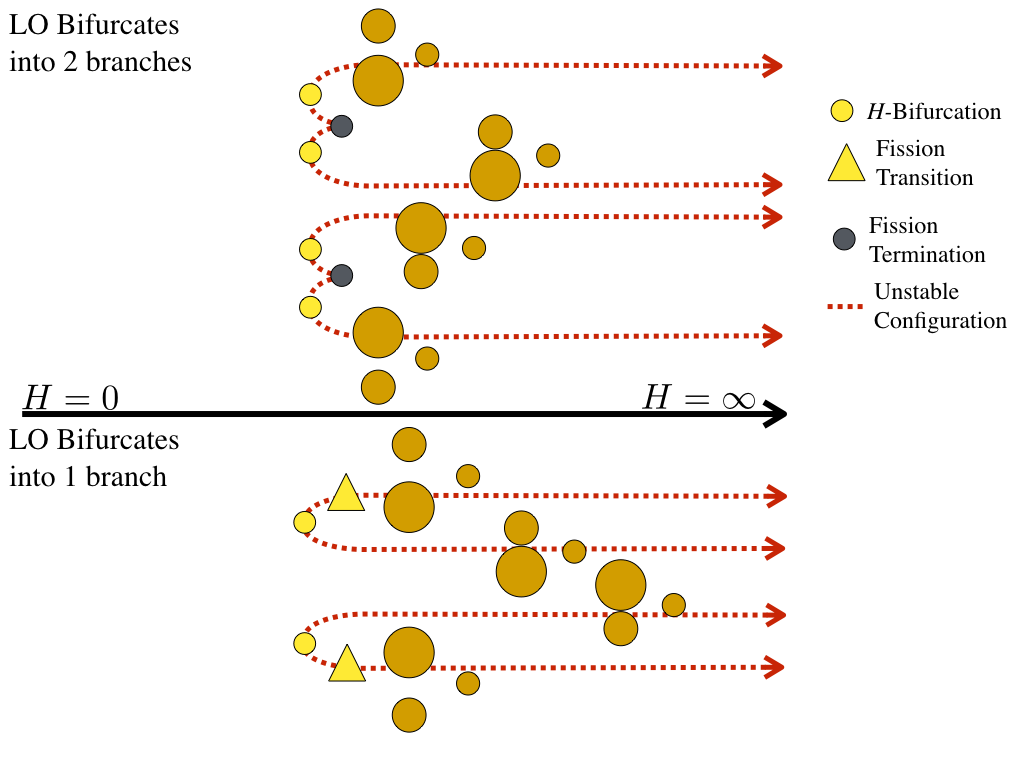}
\caption{Bifurcation diagram showing the pathways to the LO configuration.  }
\label{fig:bif4}
\end{figure}

\clearpage

\section{Problem Statement}

\subsection{The Full Body Problem}

A full body problem is defined as a set of $N$ rigid bodies that gravitationally attract each other and which have a finite density mass distribution, meaning that there are specific limits on how close they can come to each other \cite{scheeres_F2BP, scheeres_minE, scheeres_minE_chap}. The description of such a system can be directly incorporated into a Lagrangian framework where the coordinates ${\bf Q} = \left\{ Q_i; i = 1, 2, \ldots, 6N\right\}$ denote the absolute Cartesian coordinates of the bodies and the Euler angles that orient the bodies in space. The rigid body constraints place restrictions of the form $|Q_i - Q_j| \ge D_{ij}({\bf Q})$ on the system. The dynamics of the system can be described by a total Kinetic Energy and Gravitational Potential Energy and as it is an isolated system will conserve its total angular momentum, denoted as $\bfm{H}$, and can conserve its total energy, denoted as $E$, if internal dissipative forces are excluded. 

The linear momentum can be removed to reduce the system to $6(N-1)$ coordinates that are purely relative to each other and an additional 3 degrees of freedom that orient the entire system with respect to inertial space. Denote the relative coordinates as ${\bf q} = \left\{ q_i; i = 1, 2, \ldots 6(N-1)\right\}$, noting that these can always be transformed to locally reformulate the constraints as $q_i \ge D_i({\bf q})$. 

Thus, for the Full 3-Body Problem we have 12 twelve relative degrees of freedom between the three bodies. Of these only three are required to specify the relative positions of the bodies. The additional 9 correspond to each of the bodies having 3 degrees of freedom to orient themselves relative to the position configuration of the bodies. As we take the three bodies to be spheres, their relative orientation is not tracked, although we must still account for their rotational angular momentum and kinetic energy. Thus, for our purposes, our problem can be specified with only 3 degrees of freedom, plus the overall orientation of the system with respect to inertial space. 

Before continuing we define the amended potential, which plays a fundamental role in the following. 
\begin{definition}{\bf Amended Potential }
The Amended Potential is defined as the function ${\cal E} = \frac{H^2}{2 I_H} +  \mathcal{U} $ where $H$ is the total angular momentum of the system, $I_H$ is the moment of inertia of the total system taken about a principal axis of the system, in general about the rotation axis $\hat{\bfm{H}}$ which points in the direction of the total angular momentum vector, and $ \mathcal{U} $ is the gravitational potential energy of the system. The terms $I_H$ and $ \mathcal{U} $ are functions only of the relative positions and attitudes of the bodies, and for $I_H$ their orientation relative to $\hat{\bfm{H}}$. The gradients of the Amended Potential with respect to the degrees of freedom equal the force exerted on that degree of freedom when at an equilibrium or resting configuration (\cite{arnold1988mathematical}, pp 66-67). 
\end{definition}

\subsection{Spherical Full 3-Body Problem Statement}

Consider three bodies, ${\cal B}_i, i = 1,2,3$, each of which is a sphere of radius $R_i$ and, for convenience, assumed to have a common density $\rho$. 

%For definiteness we assume $R_1 \ge R_2 \ge R_3$. Thus, the masses of each sphere equals $M_i = 4\pi \rho R_i^3 / 3$, and $M_1 \ge M_2 \ge M_3$. 

The positions of these bodies can be denoted in $\mathbb{R}^3$ by Cartesian position vectors $\bfm{r}_i$. The relative positions of these bodies are denoted as $\bfm{r}_{ij} = \bfm{r}_j - \bfm{r}_i$ and have the fundamental rigid body constraint $| \bfm{r}_{ij} | \ge (R_i + R_j)$ for $i\ne j$. This lower bound, due to the bodies having finite density, is what enables resting equilibria to occur. Each of the spheres can carry angular momentum in their spin rate, although due to their symmetry the specific orientation of these spheres are arbitrary in any frame. Thus, the internal relative configuration space of the system, $\bfm{q}$, is completely specified by only three quantities
\beq
	\bfm{q} & = & \left\{ r_{12}, r_{23}, r_{31} \ | \ r_{ij} \ge (R_i + R_j)  \ \& \ |r_{ij} - r_{jk}| \le r_{ki} \le |r_{ij} + r_{jk}| \right\}
\eeq

While the configuration of the system is uniquely defined by these distances, not all distances are allowable. This means that there are geometric constraints between some of the distances independent of the finite density assumption. Thus it is sometimes easier to define a unique configuration where the restriction is clearly obvious. One such is to specify the distances between two of the bodies and the angle between these two bodies centered on the third body (see Fig.\ \ref{fig:config}). Thus, denoting the bodies with the unique indexing $i, j, k$, the configuration can be specified as 
\beq
	\bfm{q} & = & \left\{ r_{ij}, r_{jk}, \theta_{ki} \ | \ r_{lm} \ge (R_l + R_m) \right\}
\eeq
where the final distance $r_{ki}$ can be explicitly computed from the cosine rule: 
\beq
	r_{ki}^2 & = & r_{ij}^2 + r_{jk}^2 - 2 r_{ij} r_{jk} \cos\theta_{ki} \label{eq:d31}
\eeq
Note that the angle $\theta_{ki}$ will also have constraints placed upon it, as the associated length must satisfy $r_{ki} \ge R_i + R_k$. These two expressions of the third degree of freedom, $\theta_{ki}$ or $r_{ki}$, will be used equivalently. 

\begin{figure}[h!]
\centering
\includegraphics[scale=0.4]{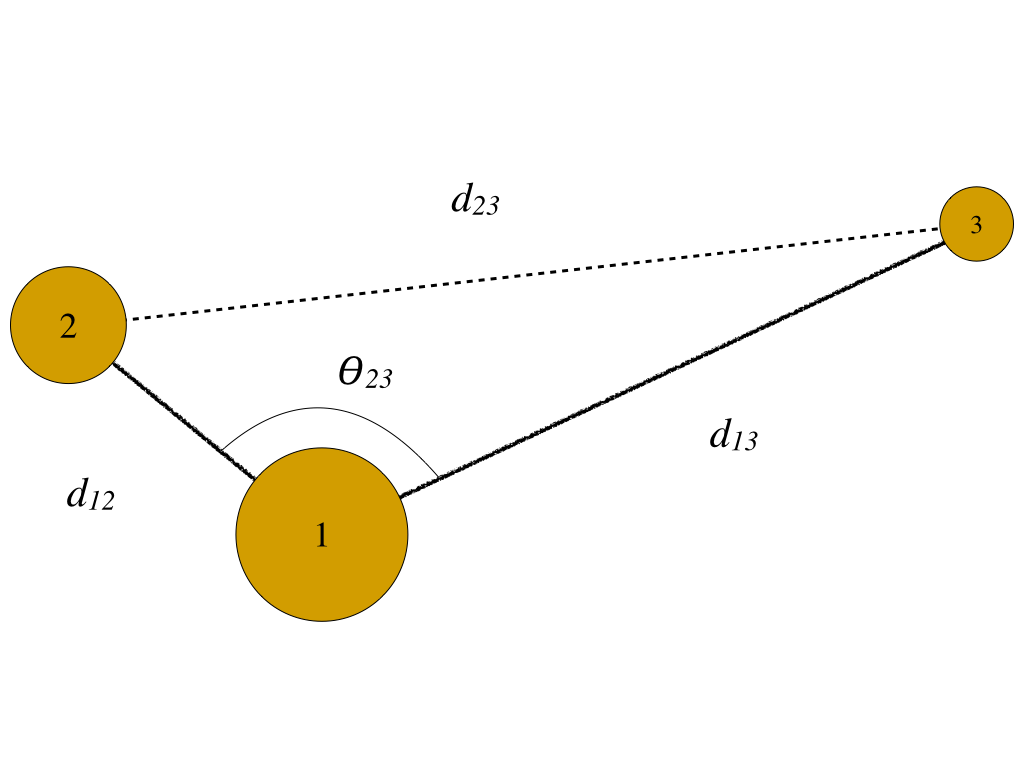}
\caption{Configuration of the system.}
\label{fig:config}
\end{figure}

There are additional degrees of freedom of the triad of bodies with respect to inertial space, which can be tracked by the unit vector of the angular momentum, $\hat{\bfm{H}}$, which are briefly discussed later. 

The gravitational potential is equivalent to the 3-body point mass potential due to the symmetry of spherical mass distributions.
\beq
	\mathcal{U} & = & - {\cal G} \left[ \frac{M_1 M_2}{r_{12}} + \frac{M_2 M_3}{r_{23}} + \frac{M_3 M_1}{r_{31}} \right] 
\eeq

The moment of rotational inertia of each sphere is equal to $2 M_i R_i^2 / 5$. For a given placement of the three masses, the total inertia dyad of the system can be constructed as
\beq
	\bfm{I} & = & \frac{1}{M_1 + M_2 + M_3} \sum_{i=1}^2 \sum_{j=2}^3 M_i M_j \left[ r_{ij}^2 \bfm{U} - \bfm{r}_{ij}  \bfm{r}_{ij} \right]  + \sum_{i=1}^3 \frac{2 M_i R_i^2}{5}  \bfm{U}
\eeq
where $\bfm{U}$ is the identity dyad. Note that this form uses the Lagrange Identity and assumes that the center of mass is nominally at a zero point. The inertia matrix is orientable, but for the amended potential only its orientation relative to the constant angular momentum vector direction, $\hat{\bfm{H}}$, is needed. Dotting the dyad on both sides by this unit vector yields 
\beq
	I_{H} & = & \frac{1}{M_1+M_2+M_3} \sum_{i=1}^{2}\sum_{j=2}^3 M_i M_j \left[ r_{ij}^2 - (\hat{\bfm{H}}\cdot\bfm{r}_{ij})^2 \right] + \sum_{i=1}^3 \frac{2 M_i R_i^2}{5}  
\eeq

The principal moments of inertia for a three point-mass system have the following relation: $I_{max} = I_{int} + I_{min}$. Furthermore the maximum moment of inertia of the point masses will always be perpendicular to the plane that the three bodies form \cite{greenwood_dynamics}. Thus, with our assumption that the body spins about its maximum moment of inertia, the quantities $(\hat{\bfm{H}}\cdot\bfm{r}_{ij}) = 0$ and the  moment of inertia simplifies to
\beq
	I_{H} & = & M_1 M_2 r_{12}^2 + M_2 M_3 r_{23}^2 + M_3 M_1 r_{31}^2 + I_S \\
	I_S & = & \frac{2}{5} M_1 R_1^2 + \frac{2}{5} M_2 R_2^2 + \frac{2}{5} M_3 R_3^2
\eeq
For rotation about the intermediate and minimum moments of inertia, we note that $I_H$ will be strictly less than or equal to this above value, with equality between the intermediate and maximum only occurring when the minimum moment of inertia of the three particles (without the rigid sphere contributions) is 0. The maximum moment of inertia of the point masses can never be zero, due to the finite size of the particles. 

%With these definitions, and an assumed value of total angular momentum for the system, $H$, the amended potential is defined as
%\beq
%	{\cal E} & = & \frac{ H^2 }{2 I_H} + \mathcal{U} 
%\eeq
%with an internal configuration space $\bfm{q}$. 

\subsection{Normalization}

To simplify the discussion, normalize the system with a length and a mass scale. The length scale used is $R_T = R_1 + R_2 + R_3$, while the mass scale is $M_T = M_1 + M_2 + M_3$. Denote $m_i = M_i / M_T$, $r_i = R_i / R_T$, and $d_{ij} = r_{ij} / R_T$. In normalized coordinates the fundamental quantities take on the values
\beq
	\mathcal{U} & = & - \left[ \frac{m_1 m_2}{d_{12}} + \frac{m_2 m_3}{d_{23}} + \frac{m_3 m_1}{d_{31}} \right] \\
	I_{H} & = & m_1 m_2 d_{12}^2 + m_2 m_3 d_{23}^2 + m_3 m_1 d_{31}^2 + I_S \label{eq:IH} \\
	I_S & = & \frac{2}{5} m_1 r_1^2 + \frac{2}{5} m_2 r_2^2 + \frac{2}{5} m_3 r_3^2
\eeq
with the angular momentum being normalized by the dividing factor $\sqrt{{\cal G} M_T^3 R_T}$ and the energy normalized by the dividing factor ${\cal G} M_T^2 / R_T $. For both $H$ and $E$ the same notational designation is kept for the normalized values.  

The normalizations provide two identities:
\beq
	r_1 + r_2 + r_3 & = & 1 \\
	m_1 + m_2 + m_3 & = & 1
\eeq
There are also fundamental relationship between the $r_i$ and the $m_i$, assuming constant density. 
\beq
	m_i & = & \frac{r_i^3}{r_1^3 + r_2^3 + r_3^3} \label{eq:massnorm} \\
	r_i & = & \frac{m_i^{1/3}}{m_1^{1/3} + m_2^{1/3} + m_3^{1/3}}
\eeq

\subsection{Parameterization of the Problem}

Any given variant of the F3BP can be identified with a point in a compact, 2-dimensional triangle, using either the masses or the radii. Plotting the radii $r_1, r_3$ or masses $m_1, m_3$ along two perpendicular axes each of them can only take values between 0 and 1, and that furthermore they will be bounded by a diagonal defined by $r_1 + r_3 \le 1$ or $m_1 + m_3 \le 1$. On the boundary of this equality $r_2 = m_2 = 0$. If, instead, a diagonal defined by $r_1+r_3 = r_{13} < 1$ or $m_1+m_3 = m_{13} < 1$ is drawn, then the value of the second radius or mass will equal $r_2 = 1 - r_{13}$ or $m_2 = 1 - m_{13}$. This also lends itself to a graphical description, shown in Fig.\ \ref{fig:triangle} for the masses. Every point within this triangle defines a unique F3BP in terms of its relative masses. In \cite{scheeres_minE} only the relative equilibria for the point $1/3, 1/3, 1/3$ was studied. This paper studies the bifurcation structure across the entire region, however due to the symmetry of the problem the study can be restricted to a specific region only. To that end, consider the restrictions
\beq
	0 \le m_3 \le m_2 \le m_1 \le 1 \\
	0 \le r_3 \le r_2 \le r_1 \le 1 
\eeq
This region is shaded in Fig.\ \ref{fig:triangle}. There are 5 other equivalent triangles defined by reordering the different inequalities given above. The approach taken will be to exhaustively study all possible relative equilibria in the denoted region, the results of which can then be easily applied to all other regions. 

\begin{figure}[h!]
\centering
\includegraphics[scale=0.45]{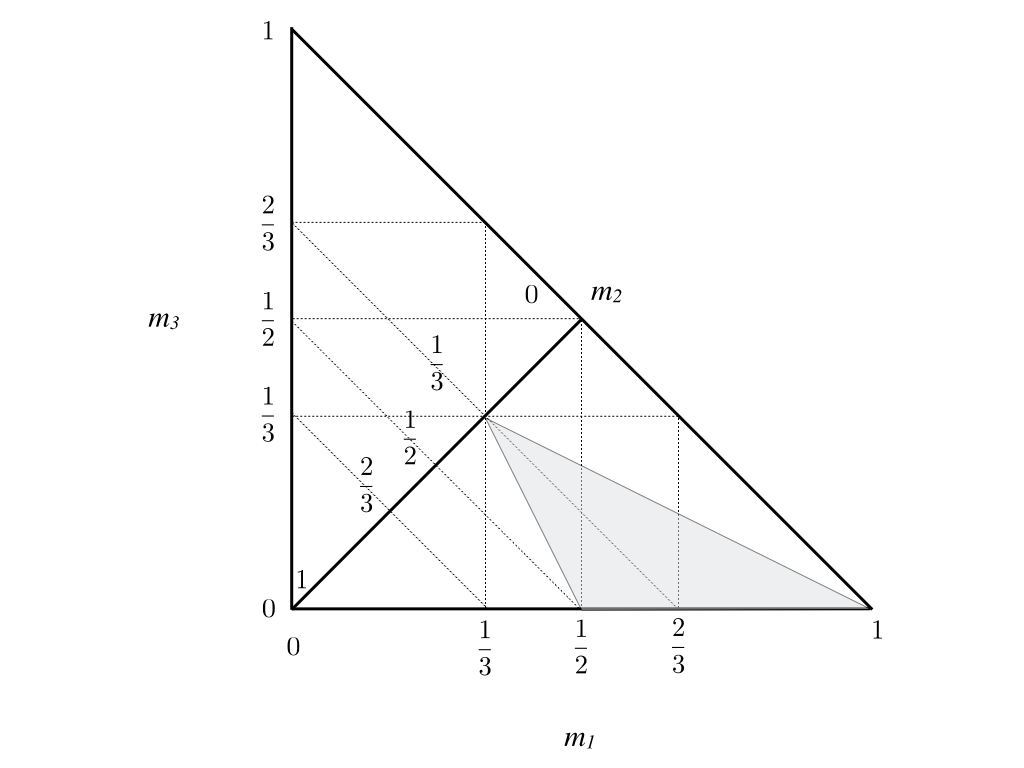}
\caption{Triangle defined for the masses with the region of study shaded.}
\label{fig:triangle}
\end{figure}

With this convention, there are additional constraints for the masses and radii.
\beq
	\frac{1}{3} \le (r_1, m_1) \le 1 \\
	0 \le (r_3, m_3) \le \frac{1}{3}  \\
	0 \le (r_3, m_3) \le (r_2, m_2) \le \frac{1}{2}
\eeq

Previous research has exhaustively explored the bifurcation structure and properties for two general cases along the boundary of this triangle. One is at the point (1/3, 1/3, 1/3), when all masses and sizes are equal \cite{scheeres_minE}. In this case a more limited number of relative equilibria were found with a less complex bifurcation structure. The other case is for $m_3 = 0$, in essence just considering the two mass case with $0 \le m_2 \le 1/2 \le m_1 \le 1$, along the base of the triangle \cite{scheeres_minE}.  In this region the number of relative equilibria are also much fewer and the bifurcation structure less complex. 

For the problem we study, the spherical 3-body problem, we can easily just consider the planar motion of the system, with rotation occurring about a principal moment of inertia of the system. We note that the spheres contribute to the system's total angular momentum but have the same moment of inertia about any axis. In general we will assume that the system rotates about the maximum moment of inertia, but will justify this later. 

\section{Background and Supporting Results}

A few definitions and supporting Lemmas are stated for use in this paper. Some of these are classical results while others have been considered more recently \cite{scheeres_minE, scheeres_minE_chap}, thus the proofs are only briefly reviewed to point out their salient features. Specific results for our current analysis are worked out in detail. 

\begin{lemma}
\label{lemma:1} 
The Total Energy of the system is conserved in the absence internal dissipation and equals $E = T_r + {\cal E}$, where $E$ is the total energy and $T_r$ is kinetic energy of the system components relative to each other, evaluated in the rotating frame with inertial angular velocity $\bfm{H} / I_H$. 
\end{lemma}

\begin{proof}
For rotation about a principal axis of the system, ${\cal E}$ equals the amended potential as introduced by Smale \cite{smaleI, smaleII}, and specifically considered by Arnold for the 3-body problem in \cite{arnold1988mathematical}, pp. 66-67. For a system rotating about its principal axis the proofs in \cite{scheeres_minE_chap} apply, showing that the amended potential arises from a Routh reduction of the system. The Routhian is shown to have a Jacobi integral, which is identical to the total energy of the system.
\end{proof}

\begin{lemma}
\label{lemma:2} 
The total energy of the system is strictly bounded from below by the amended potential: $\mathcal{ E} \le E$. If $E = \mathcal{ E}$, then $T_r = 0$. If the system is momentarily stationary ($T_r = 0$) and spins about a principal axis of inertia of the system, then $\mathcal{E} = { E}$. Thus the inequality is sharp and the lower bound can be achieved. 
\end{lemma}

\begin{proof}
Lemma \ref{lemma:2} is proven in \cite{scheeres_minE_chap}. The proof establishes the inequality using the Cauchy Inequality applied to the angular momentum, and shows that it is sharp through direct construction. 
\end{proof}

Another important feature of the system involves the existence of minimum energy configurations for full body systems. The following lemma establishes the existence of minimum energy states for all values of angular momentum. This result also provides the fundamental motivation for the current study. 
\begin{lemma}
For a finite density distribution, the amended potential ${\cal E}$ has a global minimum for all values of angular momentum $H$.
\end{lemma}

\begin{proof}
The proof is given in \cite{scheeres_minE} and involves showing that ${\cal E}$ is compact and bounded over all possible values of the configuration space. This requires the finite density assumption, as this blocks individual point masses from coming arbitrarily close to each other. If body $i$ escapes to $\infty$ relative to bodies $j$ and $k$, then the amended potential takes on the value $\mathcal{E} = \mathcal{U}_{jk}$, and remains bounded in the interval $[- m_j m_k / (1-r_i), 0]$. If all three bodies escape to $\infty$ relative to each other then $\mathcal{E} = 0$. 
\end{proof}

Given the definition of the amended potential and its properties relative to the total energy of the system, the relative equilibrium and energetic stability can be defined. Following this conditions under which these are satisfied are stated. 

\begin{definition}{\bf Relative Equilibrium}
A given configuration $\bfm{q}^*$ is said to be a ``Relative Equilibrium'' if its internal kinetic energy is null ($T_r = 0$), meaning that ${\cal E} = E$ at an instant, and if it remains in this state over at least a finite interval of time. 
\end{definition}

\begin{definition}{\bf Energetic Stability}
A given relative equilibrium $\bfm{q}^*$ is said to be ``Energetically Stable'' if any equi-energy deviation from that relative equilibrium requires a negative internal kinetic energy, $T_r < 0$, meaning that this motion is not allowed.
\end{definition}

\begin{lemma}
\label{lemma:3}
Consider a system with an amended potential ${\cal E}$ as defined above with $n$ degrees of freedom, $m$ of which are activated in such a way that only the variations $\delta q_j \ge 0$, $j = 1, 2, \ldots, m$ are allowed. The degrees of freedom $q_i$ for $m < i \le n$ are free. 

The Necessary and Sufficient conditions for a system in a configuration $\bfm{q}^*$ to be in a relative equilibrium are that at this configuration:
\begin{enumerate}
\item
$T_r = 0 $
\item
${\cal E}_{q_j} \ge 0 \ \forall \ 1 \le j \le m$
\item
${\cal E}_{q_i} = 0 \ \forall \ m < i \le n $
\end{enumerate}

The Necessary and Sufficient conditions for a system in a relative equilibrium to be energetically stable are that:
\begin{enumerate}
\item
$ {\cal E}_{q_j} > 0 \ \forall \ 1 \le j \le m$
\item
$ \left[ \frac{\partial^2 {\cal E}}{\partial q_i \partial q_k} \right] > 0 \ \forall \ m < i,k \le n$
\end{enumerate}

\end{lemma}

\begin{proof}
The proof of Lemma \ref{lemma:3} is found in \cite{scheeres_minE_chap}. It relies on taking variations of the amended potential, asserting the principle of conservation of energy and using the general form of the Lagrange equations of motion for the full body system. 
\end{proof}

Next a few results of relevance for the system are stated regarding bifurcation of relative equilibria and their stability. 
Specific example bifurcations and their properties are shown in Fig.\ \ref{fig:bif_ex}

\begin{figure}[h!]
\centering
\includegraphics[scale=0.35]{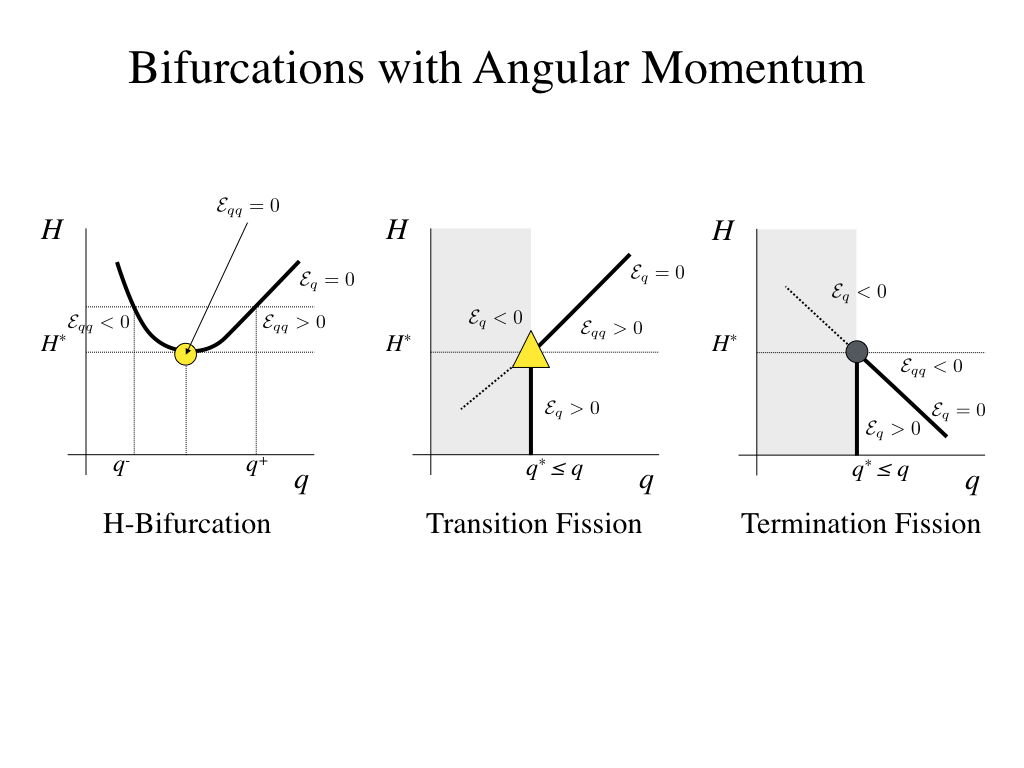}
\caption{Examples of bifurcations of interest in this problem, and their stability properties.  }
\label{fig:bif_ex}
\end{figure}

%\begin{definition}
%{\bf Asymmetric Bifurcation}
%A bifurcation of two relative equilibria which follow asymmetric paths relative to each other at changing values of angular momentum. 
%\end{definition}

\begin{definition}
{\bf Symmetric Bifurcation}
A bifurcation of two relative equilibria which follow a symmetric path relative to each other about a reflection line at changing values of angular momentum. 
\end{definition}

\begin{definition}
{\bf H-Bifurcation}
An H-Bifurcation occurs when, under increasing angular momentum, a pair of relative equilibria appear in a degree of freedom $q$ that is not at a constraint. At its first appearance there must be a degeneracy of the form $\mathcal{E}_{qq} = 0$ that will generically disappear under increasing angular momentum. 
\end{definition}

\begin{definition}
{\bf Fission} A collection of bodies in a relative equilibria with at least one active constraint is said to ``fission'' if, under an increase in angular momentum, the active constraint is released, meaning that a free relative equilibria intersects with it. Following fission the body may either transition into a new relative equilibrium without that active constraint or may no longer lie in any relative equilibrium associated with that configuration. 
\end{definition}

\begin{definition}
{\bf Termination Fission}
A fission bifurcation where the relative equilibria disappear at higher values of angular momentum.
\end{definition}

\begin{definition}
{\bf Transition Fission}
A fission bifurcation where the relative equilibria continues with its constraint inactive at higher values of angular momentum. 
\end{definition}

Now a particularly useful lemma is proven, which enables us to relate the stability of equilibrium points to how their coordinate changes as a function of angular momentum $H$. 

\begin{lemma}
\label{theorem:X}
Assume a relative equilibria exists with a 1-1 relationship between a single degree of freedom $q$ and the angular momentum $H$, meaning that along the local family of relative equilibria as the angular momentum is changed only the degree of freedom $q$ changes. Then $\mathrm{sign}( \mathcal{E}_{qq} ) = \mathrm{sign}( \partial H/\partial q )$ at the relative equilibria. Thus, if $\partial H / \partial q < 0 $ the equilibrium point will be energetically unstable and if $\partial H / \partial q > 0 $ it could be stable, depending on the other degrees of freedom. 
\end{lemma}

\begin{proof}
From the lemma statement it can be assumed that all other degrees of freedom lie in a relative equilibrium condition independent of the local value of $H$. Given this, assume that there exists a value $q^*$ such that the scalar equation $\left. \mathcal{E}_q\right|_* = - H^2 / (2 I_H^2) I_{Hq} + \mathcal{U}_q = 0$. This can be solved for $H^* = \left.\left(I_H \sqrt{ 2 I_{Hq} / \mathcal{U}_q}\right)\right|_*$, where the righthand side is a function of $q^*$ and by assumption is non-singular. Now consider a neighboring relative equilibrium at a different value of $H$ and hence $q$, with the values defined locally by the expansion
\beqs
	\mathcal{E}_q( H^* + \Delta H, q^* + \Delta q) & = & 0 + \left.\mathcal{E}_{qH}\right|_* \left.\frac{\partial H}{\partial q}\right|_* \Delta q + \left.\mathcal{E}_{qq}\right|_* \Delta q + \ldots 
\eeqs
Setting this to zero and solving for an arbitrary $\Delta q$ yields
\beqs
	\left.\mathcal{E}_{qq}\right|_* & = & - \left.\mathcal{E}_{qH}\right|_* \left.\frac{\partial H}{\partial q}\right|_*
\eeqs
However, from the defining equation for $\mathcal{E}_q$ given above, it is seen that $\mathcal{E}_{qH} = - H / I_H^2 I_{Hq}$, where $I_{Hq} > 0$ by inspection of Eqn.\ \ref{eq:IH}. Thus, the sign of $\left.\mathcal{E}_{qq}\right|_*$ equals the sign of $\left.\frac{\partial H}{\partial q}\right|_*$. 
\end{proof}

\begin{corollary}
\label{corrolary:X}
For an increasing angular momentum $H$, a free relative equilibria that ends in a Termination Fission is always unstable in the degree of freedom $q$. Conversely, a free relative equilibria that emanates from a Transition Fission is always stable in the degree of freedom $q$. 
\end{corollary}

\begin{proof}
Assume the active constraint is defined to be $q=0$. By definition, a Termination Fission occurs when a relative equilibrium at $q^* > 0$ moves towards the general constraint $q = 0$ under increasing angular momentum. Thus $\partial H / \partial q < 0$ and from Lemma \ref{theorem:X} the relative equilibrium is unstable. Conversely, a Transition Fission occurs when a relative equilibrium at $q^* \ge 0$ moves away from the general constraint $q = 0$ under increasing angular momentum. Thus $\partial H / \partial q > 0$ and from Lemma \ref{theorem:X} the relative equilibrium is energetically stable in the degree of freedom $q$, although it may be unstable in other degrees of freedom. 
\end{proof}

Finally, we end with a Lemma on the rotation axis that a stable configuration must have.

\begin{lemma}
Any relative equilibrium configuration not rotating about the maximum moment of inertia of the body will be energetically unstable. 
\end{lemma}
\begin{proof}
If a body is in a relative equilibria it must rotate about a principal moment of inertia. It can then be treated as a rigid body, at least up to first order variations in its internal configuration and inertial orientation. If it is not rotating about its maximum moment of inertia, it must be rotating about its intermediate or minimum moment of inertia.  The relevant total energy of the function system then equals $H^2 / (2 I_i)$ where $I_i$ is a principal moment of inertia (here ignoring internal variations). From the classical Poinsot construction the body will be at a saddle point of the energy function if rotating about the intermediate axis and will be at a local maximum of the energy function for rotating about the minimum axis. In either case, the rigid body rotation is not stable in the energetic sense as it can depart from this rotation axis while conserving energy with an increase in kinetic energy. 
\end{proof}
Due to this result, we only consider rotation about the maximum moment of inertia, which will always lie in the plane containing the three bodies. In the degenerate case where the bodies are in a line, the system will rotate perpendicular to its line of symmetry. 

With these Definitions, Lemmas and Corollaries stated, the relative equilibria and stability of the Full 3-body problem can be established.

\section{Existence, Stability and Bifurcation of Relative Equilibria}

In this section, having stated the theorem and developed the necessary background, the detailed proof of Theorem 1 is now given. 

\begin{proof}

To systematically explore the existence, stability and bifurcation of the relative equilibria the systems with different numbers of degree of freedom constraints activated and conditions for these to be released are considered separately. The discussion starts with all three DOF constraints activated and progressing to fewer and fewer until all degrees of freedom are not constrained. The Appendix contains the detailed partial derivatives and variation conditions of the amended potential needed for the following discussions. 

\subsection{Three Active Constraints: Lagrange Resting Configurations}

\paragraph{Existence:}
For the three constraints to be active requires that $d_{ij} = r_i + r_j = 1 - r_k$ for all of the indices. This configuration can only occur when the three bodies are mutually resting on each other. 
The relative angle between adjacent grains are then defined by
\beq
	\cos\theta_{ki} & = & \frac{ (1-r_k)^2 + (1-r_i)^2 - (1-r_j)^2 }{ 2 (1-r_k) (1-r_i) } \\
	\sin\theta_{ki} & = & \frac{\sqrt{ (1-r_k)^2(r_k-r_i r_j) + (1-r_i)^2(r_i-r_j r_k) + (1-r_j)^2(r_j-r_k r_i)} }{\sqrt{2} (1-r_k) (1-r_i)} 
\eeq
where $k, i$ take on all possible values. 
The corresponding values of $I_H$ and $\mathcal{U}$ in this configuration are
\beq
	I_H & = & m_i m_j (1-r_k)^2 + m_j m_k (1-r_i)^2 + m_k m_i (1-r_j)^2 + I_S \\
	\mathcal{U} & = & - \left[ \frac{m_i m_j}{1-r_k} + \frac{m_j m_k}{1-r_i} + \frac{m_k m_i}{1-r_j} \right]
\eeq
There are two unique orderings of the resting configuration, mirroring the orbital Lagrange configuration, which results in two distinct relative equilibria. Due to this these configurations are called the Lagrange Resting (LR) configurations. 

\paragraph{Stability:}
As this is the minimum distance for each of these bodies to achieve, this also implies that the potential energy will be minimized at this configuration. From this it can immediately be concluded that for $H=0$ this particular resting configuration is the minimum energy configuration of the system and hence is stable. 

\paragraph{Bifurcation:}
As $H$ increases from zero this system should exist as a relative equilibrium for some range of $H$, but to discover the precise range when this holds requires that the transition from three to two active constraints be investigated. Thus, as angular momentum is increased, conditions for when one of these constraints is no longer enforced is sought, meaning that one of the degrees of freedom will have an allowable variation that decreases the energy. For this configuration each of the three distances can be tested in turn to see which will lose positivity first. For the condition tested, consider the angle variation $\delta\theta_{ki} \ge 0$, keeping the other two constraints $\delta d_{ij} = \delta d_{jk} = 0$. The condition for existence (and stability) of this configuration then becomes $\delta_{\theta_{ki}}\mathcal{E} \ge 0$. Evaluating this explicitly and substituting for the equal resting conditions yields
\beq
	\delta_{\theta_{ki}} {\cal E} & = & m_k m_i \left[ - \frac{H^2}{I_H^2} + \frac{1}{(1-r_j)^3} \right] (1-r_k) (1-r_i) \sin\theta_{ki} \delta \theta_{ki}  \label{eq:3bp3}
\eeq
and substituting in for $\sin\theta_{ki}$ yields
\beq
	\delta_{\theta_{ki}} {\cal E} & = & \frac{m_k m_i}{\sqrt{2}} \left[ - \frac{H^2}{I_H^2} + \frac{1}{(1-r_j)^3} \right] \times  \\
	& & \sqrt{ (1-r_k)^2(r_k-r_i r_j) + (1-r_i)^2(r_i-r_j r_k) + (1-r_j)^2(r_j-r_k r_i)}   \delta \theta_{ki}\nonumber
\eeq
Changing $k,i$ for $i,j$ and $j,k$ only changes the items on the first line, and thus the controlling condition for the existence and stability of the Lagrange Resting configurations is 
\beq
	\frac{1}{(1-r_j)^3} & > & \frac{H^2}{I_H^2}  
\eeq
which must hold for $j = 1,2,3$. Thus the minimum value of $r_j$ gives the minimum value of $H$ for the inequality to be violated. For the specified definitions this means that $j=3$ and the loss of stability occurs about the angle $\theta_{12}$, meaning that the Lagrange Resting configuration will undergo a Termination Fission by losing contact between its two largest bodies, pivoting about the smallest grain (see Fig. \ref{fig:bif2}). 

\subsection{Two Active Constraints}

In this case two bodies rest on each other, but do not have the third contact active. A convenient way to express this is to have the two distances at their minimum value and leave the angle free, or $d_{ij} = 1 - r_k$, $d_{jk} = 1 - r_i$ with $\theta_{ki}$  only constrained by the resting limit, $d_{ki} \ge 1 - r_j$. For the moment assume that $\mathcal{E}_{d_{ij}} > 0$ and $\mathcal{E}_{d_{jk}} > 0$ (this will be checked later), and thus there is only one degree of freedom to be concerned with. Taking the first variation and substituting for the distances yields 
\beq
	\delta_{\theta_{ki}} {\cal E} & = & m_k m_i \left[ - \frac{H^2}{I_H^2} + \frac{1}{d_{ki}^3} \right] (1-r_k)(1-r_i) \sin\theta_{ki} \delta \theta_{ki}  
\eeq
which must now be identically equal to zero for the system to be in equilibrium. There are two possibilities, $\sin\theta_{ki} = 0$ or $- \frac{H^2}{I_H^2} + \frac{1}{d_{ki}^3} = 0$. Both can occur and are discussed separately, the former is called the Euler Resting configuration and the latter the Transitional Resting configuration. No assumptions about the ordering of the bodies in terms of mass, unless specified.

Each case must be tested for when the configurations cease to exist, which will occur once one of the energy variations in the active distance constraints equals zero. These will be explicitly tested for each case to determine conditions at which these equilibria no longer exist. 

\subsubsection{Euler Resting Configurations}

\paragraph{Existence:}
First consider the case when $\theta_{ki} = \pi$, noting that the angle cannot equal zero due to the finite radius constraints. Then the first variation is identically equal to zero and the bodies rest on a straight line with the ordering $i, j, k$, the system forming a relative equilibrium. These are notationally denoted as ER$ijk$, noting that configuration ER$kji$ is considered to be equivalent. Now the moment of inertia and potential energy take on the values
\beq
	I_H & = & m_i m_j (1-r_k)^2 + m_j m_k (1-r_i)^2 + m_k m_i (1+r_j)^2 + I_S \\
	\mathcal{U} & = & - \left[ \frac{m_i m_j}{1-r_k} + \frac{m_j m_k}{1-r_i} + \frac{m_k m_i}{1+r_j} \right]
\eeq
with the main difference from the Lagrange Resting (LR) configurations being that the distance $d_{ki} = 1+r_j$ due to the elongate geometry. 

\paragraph{Stability:}
Under the assumption that the two distance variations are both positive (which is true for a low enough value of $H$), the stability of this relative equilibrium can be analyzed by computing the second order variation evaluated at the resting configuration. 
\beq
	\delta_{\theta_{ki}\theta_{ki}} {\cal E} & = & - m_k m_i \left[ - \frac{H^2}{I_H^2} + \frac{1}{(1+r_j)^3} \right] (1-r_k) (1-r_i) \delta \theta_{ki}^2  \label{eq:RE_thetasq}
\eeq

Stability of this configuration occurs when $\delta_{\theta_{ki}\theta_{ki}} {\cal E} > 0$ which places a lower limit on the angular momentum for stability
\beq
	\frac{I_H^2}{(1+r_j)^3} & < & H^2
\eeq
Note that the value of angular momentum is lower than the angular momentum at which the LR configurations cease to exist. Also, the stability transition occurs when the Transitional Resting configuration conditions are satisfied for the same configuration, indicating that a bifurcation occurs. 

\paragraph{Bifurcation:}
For lower values of angular momentum the Euler Resting configuration exists, but is unstable and mimics an inverted pendulum. When the stability condition is satisfied, the system mimics a hanging pendulum and will remain stable until one of the energy distance variations becomes zero, indicating a transition from two active constraints to a single active constraint. To probe when this occurs, substitute the equilibrium condition into Eqn.\ \ref{eq:3bpij} to find
\beq
	\delta_{d_{ij}} {\cal E} & = & m_i \left\{ m_j \left[ - \frac{H^2}{I_H^2} + \frac{1}{(1-r_k)^3} \right] (1-r_k) \right. \nonumber \\
	& & \left. + m_k \left[ - \frac{H^2}{I_H^2} + \frac{1}{(1+r_j)^3} \right] (1+r_j) \right\} \delta d_{ij} 
\eeq
Setting this to be greater than or equal to zero defines when the ER relative equilibrium configuration exists, and can be solved for as a condition on angular momentum
\beq
	\left[ m_j (1-r_k) + m_k (1+r_j) \right]  \frac{H^2}{I_H^2} & \le & m_j \frac{1}{(1-r_k)^2}   + m_k \frac{1}{(1+r_j)^2}  
\eeq
The precise value of $H$ when this is first violated is discussed in a later section. For the current analysis it suffices to note that this inequality is always satisfied when the ER configurations first become stable. Substituting $(H/I_H)^2 = 1 / (1+r_j)^3$ and simplifying yields  
\beq
	0 & \le & r_j + r_k    
\eeq
which is trivially satisfied for any $j$ or $k$. It is also clear that a large enough $H$ will always be able to violate the existence condition. Generically, one of the two bodies $i$ or $k$ will separate from $j$, leaving the other body in contact and transitioning the configuration into the Euler Aligned configuration. 

\subsubsection{Transitional Resting Configurations}
\paragraph{Existence:}
When the Euler Resting (ER) configurations becomes stable, a pair of solutions that satisfy the second equilibrium condition bifurcate from or into the resting configuration. The condition in general is $H^2 = I_H^2 / d_{ki}^3$, but now the moment of inertia $I_H$ becomes a function of $\theta_{ki}$ and must change with $H$ to maintain this condition. There are two branches, $\theta_{ki} > \pi$ and $\theta_{ki} < \pi$, and these give two different orderings of the configuration -- ultimately corresponding to the two different orientations of several of the equilibrium configurations. 
%Due to the finite radius limits, there are strict inequalities on the angular momentum when these configurations can occur. 
%\beq
%	\frac{I_H}{(1+r_j)^{3/2}} & \le H \le & \frac{I_H}{(1-r_j)^{3/2}}
%\eeq
%This means that the actual distance between grains $i$ and $k$ can be adjusted to satisfy the equation $(H/I_H)^2 = 1/d_{ki}^3$ while the grains continue to contact the middle grain $j$. 
The moment of inertia and potential energy now take on the more generalized form 
\beq
	I_H & = & m_i m_j (1-r_k)^2 + m_j m_k (1-r_i)^2 + m_k m_i d_{ki}^2 + I_S \\
	\mathcal{U} & = & - \left[ \frac{m_i m_j}{1-r_k} + \frac{m_j m_k}{1-r_i} + \frac{m_k m_i}{d_{ki}} \right] \\
	d_{ki}^2 & = & (1-r_k)^2 + (1-r_i)^2 - 2 (1-r_k)(1-r_i)\cos\theta_{ki} 
\eeq

\paragraph{Stability:}
Evaluating the second variation of the energy with respect to $\theta_{ki}$ yields 
\beq
	\delta^2_{\theta_{ki}}{\cal E} & = & m_k m_i \left[ 4 m_i m_k d_{ki}^2 - 3 I_H \right] \frac{(d_{ij} d_{jk} \sin\theta_{ki})^2}{I_H d_{ki}^5} \left(\delta\theta_{ki}\right)^2
\eeq
Stability, when the configuration exists, then hinges on the sign of $4 m_i m_k d_{ki}^2 - 3 I_H$. Making the substitution from Eqn.\ \ref{eq:massnorm} the stability condition can be reduced to
\beq
	d_{ki}^2 & > & \frac{3}{r_i^3 r_k^3} \left[ \frac{2}{5}\left(r_i^3+r_j^3+r_k^3\right)\left(r_i^5+r_j^5+r_k^5\right) +
		r_i^3 r_j^3 (1-r_k)^2 + r_j^3 r_k^3 (1-r_i)^2 \right] 
		\label{eq:TRstab}
\eeq
where $1-r_j \le d_{ki} \le 1+r_j$. 

Note that the equilibrium configuration does not necessarily exist across this entire range of mutual distances.  
Specifically, the distance variation conditions must be verified for the configuration to exist. Substituting the equilibrium condition into Eqn.\ \ref{eq:3bpij} then yields the existence condition (after simplification)
\beq
	\frac{H^2}{I_H^2} & \le & \frac{1}{(1-r_k)^3}  
\eeq
where $k$ is the radius of either of the outer resting bodies. Note that the transitional resting configurations will always exist over some interval of angular momentum, as substituting the initial bifurcation conditions of $(H/I_H)^2 = 1 / (1+r_j)^3$ can be trivially shown to satisfy the above existence condition. 
Again, note that $H$ can also always be chosen large enough for the existence condition to be violated. 
There are three different possible situations to cover, investigated in detail below.

\paragraph{$i=1,j=2,k=3$} For this sequence the Transitional Resting equilibrium are unstable and migrate from the ER123 configuration (which they stabilize upon bifurcation from it) to the distance $d_{31} = 1-r_3$, at which point body 1 separates from the system. To determine instability evaluate Eqn.\ \ref{eq:TRstab} over the entire range of radius values and verify that it is never satisfied. 
To see this consider the contact conditions from Eqns.\ \ref{eq:3bpij} and \ref{eq:3bpjk}. For these conditions to hold both must be greater than or equal to zero for a positive variation in the mutual distance $d_{ij}$ and $d_{jk}$. Substituting the equilibrium condition $(H/I_H)^2 = 1/d_{31}^3$ and simplifying, the condition for existence of the TR123 configuration is that both
\beq
	d_{31} & \ge & d_{12} \\
	d_{31} & \ge & d_{23}
\eeq
For the current configuration, $d_{12} = 1-r_3 > d_{23} = 1-r_1$. Thus the controlling condition is $d_{31} \ge 1 - r_3$. Now note that $d_{31} \ge 1-r_2$ and that $1-r_3 > 1-r_2$, thus this inequality is violated prior to the TR123 configuration reaching the LR configuration, and as noted occurs once $d_{31} = 1-r_3$. See Fig.\ \ref{fig:bif1} for the evolutionary path for this situation. 

\paragraph{$i=3,j=1,k=2$} For this sequence the Transitional Resting equilibrium are unstable (determined as before) and migrate from the ER312 configuration (which they stabilize upon bifurcation from it) to the distance $d_{23} = 1-r_3$ when body 2 separates from the system. Similar to above, the condition for existence of the TR312 configuration is that 
\beq
	d_{23} & \ge & d_{31} \\
	d_{23} & \ge & d_{12}
\eeq
For the current configuration, $d_{12} = 1-r_3 > d_{31} = 1-r_2$. Thus the controlling condition is $d_{23} \ge 1 - r_3$. Now note that $d_{23} \ge 1-r_1$ and that $1-r_3 > 1-r_1$, thus this inequality is violated prior to the TR312 configuration reaching the LR configuration, and as noted occurs once $d_{23} = 1-r_3$ again.  See Fig.\ \ref{fig:bif3} for the evolutionary path for this situation. 

\paragraph{$i=1,j=3,k=2$} For this sequence the Transitional Resting equilibria exist across the range of radius limits, going from ER132 to Lagrange Resting configurations. For this configuration there are ranges of parameters for which there are stable relative equilibria. To identify these regions compare the upper inequality limit to when the distance for stability is less than the maximum distance $1+r_j$. Plotting out this region delineates the small oval region in Fig.\ \ref{fig:TR132}. For parameter values within this region the evolution of the TR132 configuration becomes more complex. Specifically, the angular momentum profile in this region is such that there are two relative equilibria defined at a given level of angular momentum, one towards the LR configuration (which is always at a local maximum of the energy and thus is unstable) and one towards the ER132 configuration (which becomes a local minimum of the energy and thus is stable). Figure \ref{fig:bif2} shows the two different pathways that can occur. 

\paragraph{Bifurcation:}
For the TR configurations which are always unstable, as the angular momentum is increased they first bifurcate into existence by stabilizing the ER configurations. Then as $H$ is increased they migrate towards more compact configurations. The TR123 and TR312 configurations then end with one of the bodies separating from the other two. The TR132 configuration migrates all the way to the LR configuration and destabilizes it, thus terminating both the LR and TR132 configurations. 

For TR configurations that can be stable, indicated in Fig.\ \ref{fig:TR132}, the sequence is different. Here as $H$ is increased an $H$-Bifurcation occurs at an angle between the minimum (or maximum) constrained value of $\theta_{12}$ and $\pi$. To show this consider the angular momentum as a function of distance $d_{12}$, $H = I_H(d_{12}) / d_{12}^{3/2}$. Taking the partial of this with respect to $d_{12}$ shows that there is a zero in the interval $1-r_3 \le d_{12} \le 1+r_3$ (meaning that $H$ takes on an extreme value) whenever the stability condition in Eqn.\ \ref{eq:TRstab} is satisfied. Further, taking the second partial of $H$ and substituting the equilibrium condition shows that this is always positive, meaning that $H$ takes on a minimum value in the interval. Thus, as $H$ is increased the two equilibria exist on either side of the minimum, with no other equilibria emerging due to the definiteness of $H$ as a function of $d_{12}$. From Lemma \ref{theorem:X} the equilibria that moves down to the LR configuration must be unstable, and thus the equilibria that moves toward the ER configuration is stable (this also agrees with the condition as formulated in Eqn.\ \ref{eq:TRstab}). 

Thus the unstable TR132 configuration continues down to the LR configuration and terminates it. The stable TR132 configuration moves toward the ER132 configuration and terminates there, stabilizing the ER132 configuration. The existence of these stable TR132 configurations is unexpected and breaks the symmetry otherwise seen in these configurations. The region where these occur correspond to grains with a nearly equal $r_1$ and $r_2$, with $r_3$ neither close to zero or to the size of the other grains. 

\begin{figure}[h!]
\centering
\includegraphics[scale=0.3]{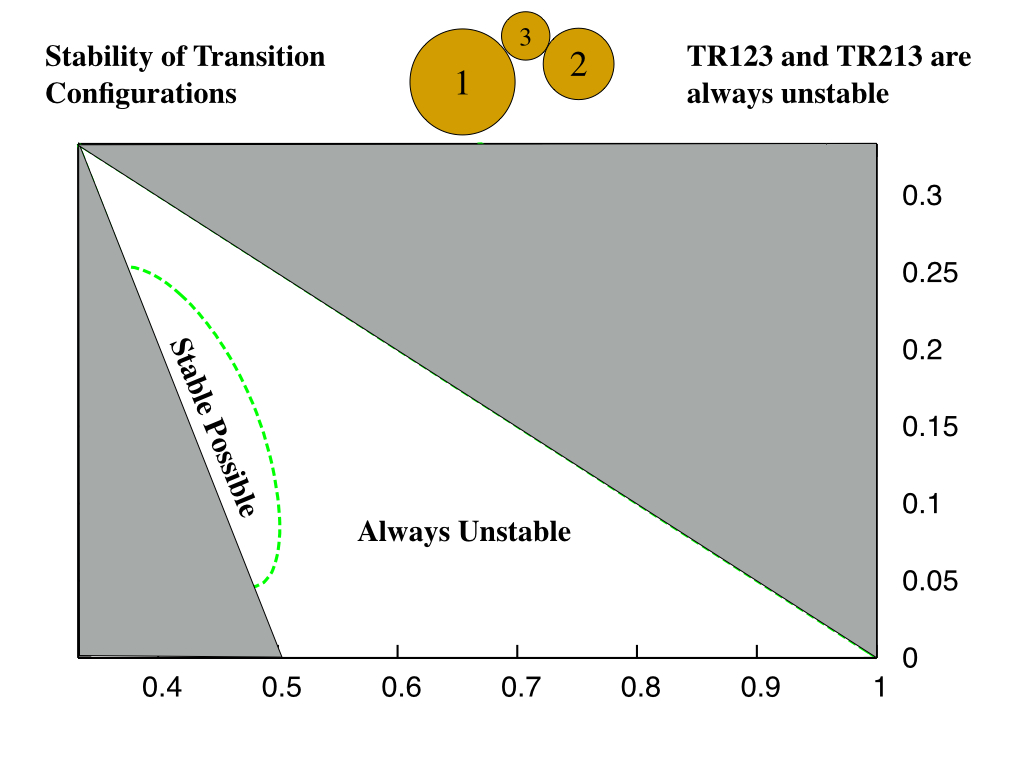}
\caption{Region where stable TR132 configurations can lie.}
\label{fig:TR132}
\end{figure}

\subsection{One Active Constraint}

Now consider relative equilibria when there is a single active constraint. In this configuration two of the bodies rest on each other, say $i$ and $j$ and thus $d_{ij}=1-r_k$, and the third body is located by the distance $d_{jk}$ and by either $d_{ki}$ or the angle $\theta_{ki}$. There are two classes of relative equilibrium solutions in this class, with the two bodies in contact either being aligned with the third body, or with their line of contact being orthogonal to the third body. The former are called the Euler Aligned configurations and the latter the transverse, or Isoceles, configurations -- the terminology arising due to the structure that these make. For these structures there are two limiting cases, one where the final active constraint separates and the other where one of the free constraints becomes activated. The former occurs when the single active constraint configurations intersect with the orbital configurations. The latter occurs when it intersects with a double active configuration. These two classes of configurations are discussed in turn. 

First it can be established that these are the only relative equilibrium configurations. Consider Eqns.\ \ref{eq:3bpjk} and \ref{eq:3bpki} in the Appendix, which both must equal zero. There are two possibilities for Eqn.\ \ref{eq:3bpki} to equal zero, either $\sin\theta_{ki} = 0$ or $H = I_H / d_{ki}^{3/2}$. Consider $\theta_{ki} = \pi$, as setting the angle to 0 is equivalent to a reordering of the bodies. Then for Eqn.\ \ref{eq:3bpjk} to equal zero the condition becomes $m_j \left( H^2/I_H^2 - 1 / d_{jk}^3\right) d_{jk} + m_i \left( H^2/I_H^2 - 1 / d_{ki}^3\right) d_{ki} = 0$. In this configuration $d_{ki} > d_{jk}$ and thus along this configuration it can never occur that $H^2 / I_H^2 = 1/d_{ki}^3$, meaning that this condition will not intersect with the $\theta_{ki} = \pi$ configuration. If $H$ is chosen such that $ I_H / d_{ki}^{3/2} < H < I_H / d_{jk}^{3/2}$ it is possible for the second condition to be satisfied, which is explored in more detail below.

The alternate condition to consider is $H^2 / I_H^2 = 1/d_{ki}^3$, with no immediate constraint on $\theta_{ki}$. Then, by substitution into Eqn.\ \ref{eq:3bpjk} yields the condition $H^2 / I_H^2 = 1/d_{jk}^3$, or $d_{ki} = d_{jk} = d$. From Eqn.\ \ref{eq:d31} and with $d_{ij} = 1 - r_k$ the condition on $\theta_{ki}$ becomes 
\beq
	\cos\theta_{ki} & = & \frac{1}{2} \frac{1-r_k}{d}
\eeq
Note that $d$ can always be chosen large enough for $\theta_{ki}$ to be well defined. 

\subsubsection{Isosceles Configurations}
\label{sec:isosceles}

\paragraph{Existence:}
The Isosceles configurations are described by having two grains in contact, nominally $i$ and $j$, the third grain $k$ in non-contact with these grains, and with the line connecting the grains $i$ and $j$ being perpendicular to the line from grain $k$ to the center of mass of grains $i$ and $j$. These are referred to as IS$ij$-$k$, with the first two indices indicating the grains in contact and the separated third index the separated grain. In terms of Eqns.\ \ref{eq:3bpij}, \ref{eq:3bpjk} and \ref{eq:3bpki}, set $d_{ij} = 1-r_k$ and $d_{jk} = d_{ki} = d$, forming an Isosceles triangle. The equilibrium condition is then simply stated as
\beq
	\frac{H^2}{I_H^2} & = & \frac{1}{d^3} \\
	I_H & = & m_i m_j (1-r_k)^2 + (m_j m_k + m_k m_i) d^2 + I_S
\eeq
Making this substitution, see that $\delta_{d_{jk}}{\cal E} = \delta_{\theta_{ki}}{\cal E} = 0$. So long as $d_{jk} \ge 1-r_i$ and $d_{ki} \ge 1-r_j$, the remaining condition for this equilibrium to be satisfied is $	\delta_{d_{ij}}{\cal E} \ge 0$, which can be simplified to the condition
\beq
	d = d_{ki} = d_{kj} & \ge & 1-r_k
\eeq

Now consider the existence of each of the possible combinations, in turn. 

\paragraph{IS12-3}
Here the grains in contact are separated by a distance $1-r_3$ and the controlling distance of the equal legs of the triangle will be $d_{31} = 1-r_2$. Note that $1-r_3 > 1-r_2$, and thus the above existence condition will be violated when $d_{31} = d_{23} = 1-r_3$, and in fact the three grains will lie at the vertices of an equilateral triangle. This condition corresponds to the intersection of IS12-3 with the orbital Lagrange configuration, LO, and terminates the IS12-3 configuration, without having grain 3 contacting the other two grains. This sequence is isolated from the previous  configurations as the three grains never come into contact and is shown in Fig.\ \ref{fig:bif4}. 

\paragraph{IS23-1}
Now the grains in contact are separated by a distance $1-r_1$ and the controlling distance of the equal legs of the triangle will be $d_{12} = 1-r_3$. Now as $1-r_3 > 1-r_1$, grains 1 and 2 will touch prior to the grains 2 and 3 separating. Once grains 1 and 2 touch the configuration matches the end-state configuration of the TR123 configuration. Thus the TR123 and IS23-1 configurations terminates, as shown in Fig.\ \ref{fig:bif1}. 

\paragraph{IS31-2} 
Now the grains in contact are separated by a distance $1-r_2$ and the controlling distance of the equal legs of the triangle will be $d_{12} = 1-r_3$ again.  Similar to before, $1-r_3 > 1-r_2$, so grains 1 and 2 will touch prior to the grains 1 and 3 separating. Once grains 1 and 2 touch the configuration matches the end-state configuration of the TR312 configuration. Thus the TR312 and IS31-2 configurations terminate, as shown in Fig.\ \ref{fig:bif3}. 

\paragraph{Stability:}
Now consider the stability of the IS configurations. The condition $\delta_{d_{ij}}{\cal E} \ge 0$ is uniformly satisfied, except for the termination of the IS12-3 configuration noted above. Thus it is just needed to test whether the joint variations of $\delta d_{jk}$ and $\delta \theta_{ki}$ are positive definite or not. For this situation one must take the second partial of the energy with respect to both of these variations, evaluated at the relative equilibrium, and test the 2$\times$2 resulting matrix for whether it is positive definite.  
\beq
	\delta^2{\cal E} & = & \left[ \begin{array}{cc} \delta d_{jk} & \delta\theta_{ki} \end{array} \right] 
		\left[ \begin{array}{cc} \frac{\partial^2{\cal E}}{\partial d_{jk} \partial d_{jk} } & \frac{\partial^2 {\cal E}}{\partial d_{jk} \partial\theta_{ki}} \\
			\frac{\partial^2 {\cal E}}{\partial\theta_{ki}\partial d_{jk} } &  \frac{\partial^2{\cal E}}{\partial\theta_{jk} \partial\theta_{jk} } \end{array} \right] 
			\left[ \begin{array}{c} \delta d_{jk} \\ \delta\theta_{ki}\end{array} \right] 
\eeq
A matrix is positive definite by Sylvester's Criterion if all of its leading principal minors are positive. A simpler, necessary condition, is that the diagonals of the matrix are all positive. 

To that end, consider the term $\frac{\partial^2{\cal E}}{\partial\theta_{jk} \partial\theta_{jk} }$ evaluated at the equilibrium condition, which can be found to equal
\beq
	\frac{\partial^2{\cal E}}{\partial\theta_{jk} \partial\theta_{jk} } & = & m_k m_i \left[ m_k\left(m_i-3m_j\right)d_{ki}^2 - 3 I_S - 3 m_i m_j (1-r_k)^2 \right] \frac{\left(d_{ij}\sin\theta_{ki}\right)^2}{d_{ki}^3 I_H}
\eeq
Note that the ordering of $i$ and $j$ does not matter, although the individual terms of the matrices may change. Thus, one can always choose to assign $i$ and $j$ such that $m_j > m_i$ to ensure that $m_i-3m_j < 0$, making the diagonal negative definite. Thus, any of the configurations will violate the necessary condition for the system to be positive definite, meaning that the Isosceles configurations are always unstable. Note that this instability mode is related to the angle $\theta_{ki}$ and not related to instability in the distance variation. Due to this, the IS family is always unstable even if it is formed from a Transition Fission. 

\paragraph{Bifurcation:}
To end, note that in Figs.\ \ref{fig:bif1}, \ref{fig:bif3} and \ref{fig:bif4} a similar specific sequence for the evolution of all of the Isosceles configurations is shown, with them appearing as an $H$-Bifurcation with one branch continuing to $\infty$ and the other terminating at a TR configuration or ending at an LO configuration. The persistence of this structure can be proven, using a similar approach as used in discussing the bifurcation of the TR132. The relation between angular momentum and distance $d$ in these configurations is the simple expression $H = I_H(d) / d^{3/2}$. It can be shown that this function has a unique minimum positive value, and thus the Isosceles configurations bifurcate into existence when the angular momentum rises above this value. Further, it can be shown that the distance at which this bifurcation occurs is always greater than the associated contact distances for this configuration. This is shown by developing a specific inequality that must be satisfied, and then checking it by computing level sets across the domain of possible radii. Doing so reveals that the bifurcation at a non-zero value of $H$ always occurs away from any of the contact termination conditions. Thus the pattern of having one branch progress towards the TR configurations and the other branch extend to large distances can be inferred. 

\subsubsection{Euler Aligned Relative Equilibria} 
\label{sec:aligned}

\paragraph{Existence:}
The Euler Aligned relative equilibria are defined by having two grains in contact and the third at a distance along the centers of mass of the two grains in contact. Again, the grains in contact are $i$ and $j$ and grain $k$ is separated. The notation for these equilibria is EA$ij$-$k$, where the order is important. Specifically, note that EA$ij$-$k$ and EA$ji$-$k$ are different, with the grain $k$ rotated 180$^\circ$ relative to the other configuration. Thus, EA$ij$-$k$ can be organized from left to right and fits with the earlier notation. There are 6 different configurations that can be considered, EA12-3, EA21-3, EA13-2, EA31-2, EA23-1, EA32-1. In all these definitions the angle $\theta_{ki} = \pi$ and $\delta_{\theta_{ki}}{\cal E} = 0$. The two remaining conditions are then $\delta_{d_{ij}} {\cal E} \ge 0$ and $\delta_{d_{jk}} {\cal E} = 0$.

For existence, solve each of these conditions for the ratio $(H/I_H)^2$ to find
\beq
	\left(\frac{H}{I_H}\right)^2 & \le & \frac{1}{m_j(1-r_k)+m_kd_{ki}} \left[ \frac{m_j}{(1-r_k)^2} + \frac{m_k}{d_{ki}^2} \right] \\
	\left(\frac{H}{I_H}\right)^2 & = & \frac{1}{m_i d_{ki} + m_j d_{jk}} \left[ \frac{m_i}{d_{ki}^2} + \frac{m_j}{d_{jk}^2} \right] 	
	\label{eq:H_EA}
\eeq
Note that $d_{ki} \ge 1 + r_j$ and $d_{jk} = d_{ki} - (1-r_k)$. 
These conditions can be combined and rewritten into a standard form
\beq
	\frac{1}{d_{ki}^3} F( m_j/m_i, d_{jk}/d_{ki} ) & \le & \frac{1}{(1-r_k)^3} F( m_k/m_j, d_{ki}/(1-r_k))
\eeq
where $F(\mu, x) = \left( 1 + \mu / x^2 \right) / \left( 1 + \mu x\right)$. Note the identity $x^3 F(\mu,x) = F(1/\mu, 1/x)$.  When this inequality is violated grains $i$ and $j$ will separate and the configuration will cease to exist. 

As a final step, define $r = d_{ki} / (1-r_k) > 1$, $\mu_{ij} = m_i/m_j$ and $d_{jk}/d_{ki} = 1 - 1/r$. Then the inequalities are written as
\beq
	F( \mu_{ji}, 1 - 1/r) & \le & F(\mu_{jk}, 1/r)
\eeq 
It can be shown (see Appendix) that $F(\mu,r)$ is monotonically decreasing in $r$ and is convex. From this it can be shown that $F( \mu, 1 - 1/r)$ is monotonically decreasing in $r$ and that $F( \mu, 1/r)$ is monotonically increasing. Thus the inequality can be crossed either 0 or 1 times and it is not needed to consider the possibility of multiple transitions in the existence of solutions. Given the well defined interval over which the parameter $r$ is defined, $r \in [(1+r_j)/(1-r_k), \infty)$,  an explicit method for determining when these conditions exist can be developed. 

First note that $\lim_{r\rightarrow\infty} F( \mu, 1 - 1/r) = 1$ and that $\lim_{r\rightarrow\infty}  F(\mu, 1/r) \sim r^2/\mu + \ldots$. Thus the inequality is always satisfied when the distance between the grains in contact, $i$ and $j$, and the separated grain $k$, is large. This holds independent of the ordering of the indices, and means that all of the EA configurations exist when the $k$th grain is sufficiently distant from the two in contact. Thus, to ascertain whether the configuration exists across all possible values of $r$ it is only needed to check the condition at the minimum radius condition. Due to the topological properties of the two functions in the inequality, if the inequality is satisfied at the minimum value of $r$, then the given configuration exists across all distances in the interval. If it is violated at the minimum value of $r$, then there exists a distance at which the configuration ceases to exist. 

Evaluating the inequality at the minimum distance $r = (1+r_j)/(1-r_k)$ yields
\beq
	F( \mu_{ji}, (1-r_i)/(1+r_j) ) & \le & F(\mu_{jk}, (1-r_k)/(1+r_j))
\eeq 
If this inequality is confirmed, then the configuration EA$ij$-$k$ exists across the whole domain and, by swapping indices $i$ and $k$, that then the configuration EA$kj$-$i$ does not exist by definition. Conversely, if the inequality for EA$ij$-$k$ is not confirmed at the lower limit, then the configuration EA$kj$-$i$ is trivially confirmed. The following discussions will assume that the configuration EA$ij$-$k$ exists all the way to contact, and thus that the configuration EA$kj$-$i$ does not and terminates at a finite distance from contact. 

This means that whenever one EA configuration exists down to the Euler Resting configuration, that the alternate EA configuration does not, and terminates at a finite separation. The termination of the conjugate configuration occurs when that configuration intersects with the conditions for the orbital Euler configuration EO$ijk$, as by definition at termination $\delta_{d_{ji}}{\cal E} = 0$ by default and $\delta_{d_{kj}}{\cal E} = 0$ due to the contact constraint vanishing. 

With these results in hand, the realms where the different Euler Aligned configurations exist can be evaluated. To do this, plot the level sets of the function
\beq
	F(\mu_{jk}, (1-r_k)/(1+r_j)) - F( \mu_{ji}, (1-r_i)/(1+r_j) ) & = & 0
\eeq
As these functions are analytical and have no singularities, there are no computational issues with evaluating these level sets. The zero line delineates where a transition in the existence of these configurations occurs. In the region where the difference is positive, the EA$ij$-$k$ configuration exists down to contact, while in the region where the difference is negative, the EA$kj$-$i$ configuration exists down to contact. These distinctions are important as they control which grain will separate from an Euler resting configuration when angular momentum is increased. In the following the plots of these zero lines are displayed for the different possible configurations. 

\paragraph{EA12-3 and EA32-1} 
Figure \ref{fig:EA123} shows a plot of the level set of the inequality for the ordering 123, showing that there exists a region where the EA12-3 configuration exists down to the ER123 configuration, and where the EA32-1 configuration exists down to the ER123 configuration. The former exists in the region where the grains 2 and 3 are more similar sized, and the latter where the grains 1 and 2 are more similar sized. Which side of the line that configuration lies determines how the configuration will fission when it terminates. Figure \ref{fig:bif1} shows the different bifurcation pathways that occur. 

\begin{figure}[h!]
\centering
\includegraphics[scale=0.25]{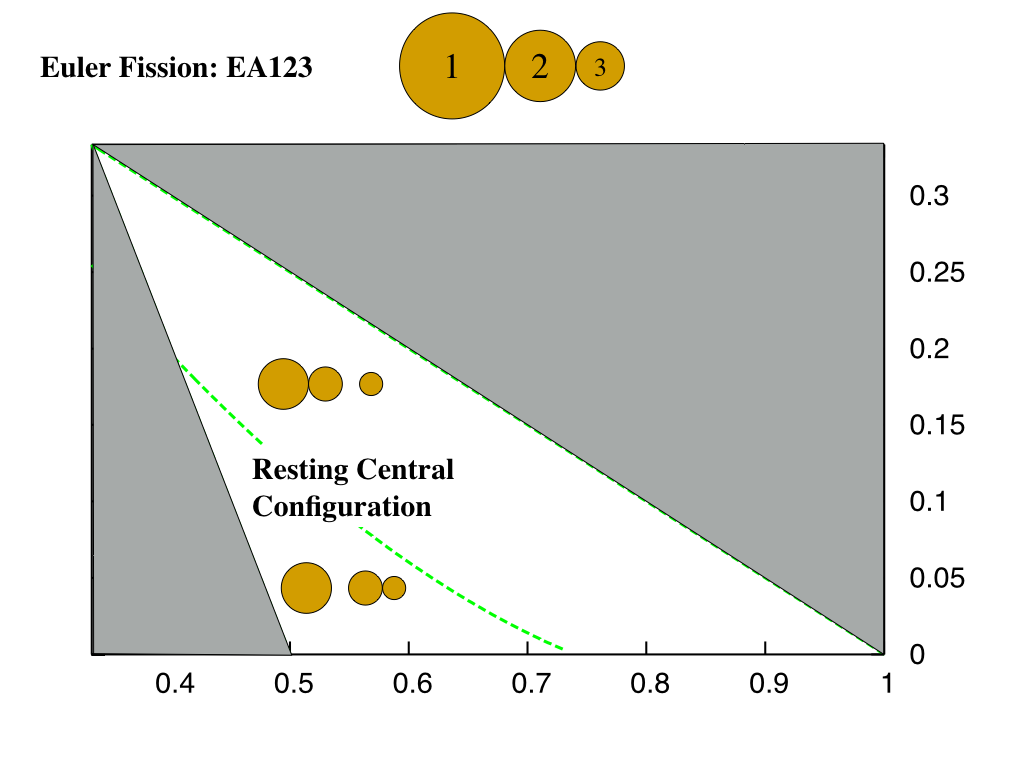}
\caption{Fission chart for the ER123 configuration. For masses to the right of the line ER123 configuration fissions by having the smallest mass separate and transitioning into the EA$12$-$3$ configuration. For masses to the left of the line the ER123 configuration fissions by having the largest mass separate and transitioning into the EA32-1 configuration.}
\label{fig:EA123}
\end{figure}

\paragraph{EA13-2 and EA23-1} 
Figure \ref{fig:EA132} shows a plot of the level set of the inequality for the ordering 132. Here, only the EA13-2 configuration exists down to the ER132 configuration, and thus the EA23-1 configuration always terminates at a finite distance. Not shown here explicitly is that at the left border, where $r_1 = r_2$, the two conditions are equivalent due to symmetry and both EA configurations extend down to the ER132 configuration. Figure \ref{fig:bif2} shows the different bifurcation pathways that occur. 

\begin{figure}[h!]
\centering
\includegraphics[scale=0.25]{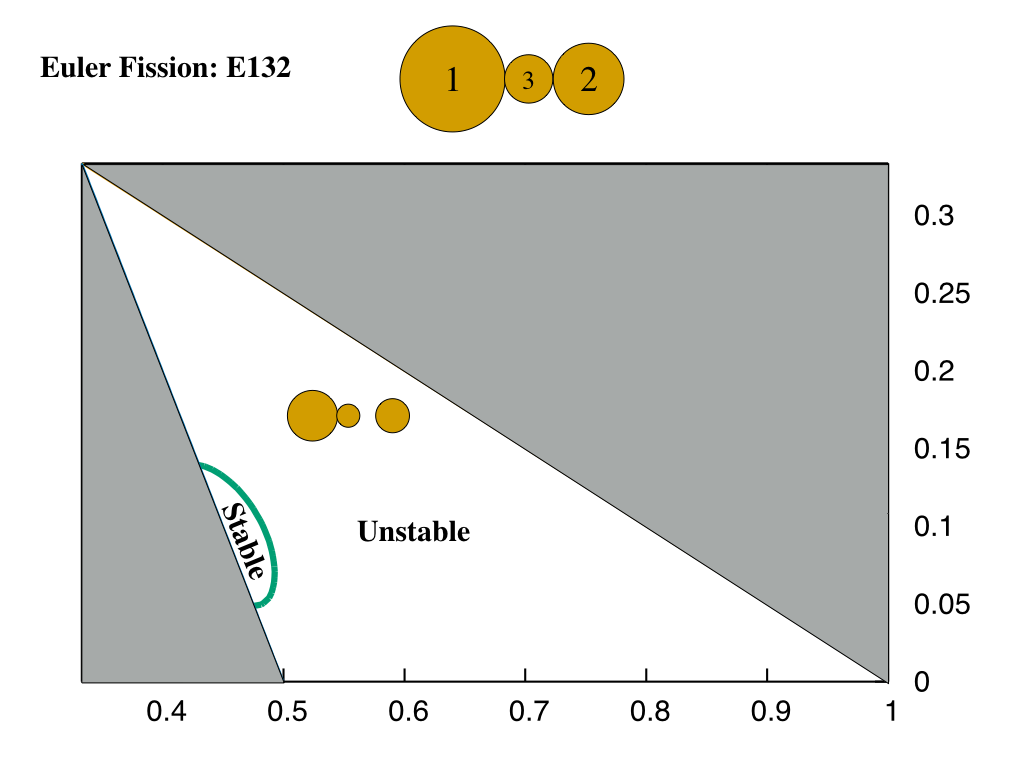}
\caption{Fission chart for the ER132 configuration. For this configuration the intermediate mass always separates, transitioning into the EA13-2 configuration. }
\label{fig:EA132}
\end{figure}

\paragraph{EA31-2 and EA21-3}
Figure \ref{fig:EA312} shows a plot of the level set of the inequality for the ordering 312. There are two regimes again. When the grains are relatively equal in size the configuration EA21-3 continues down to ER312. Away from this geometric region, however, configuration EA31-2 continues down to ER312. Figure \ref{fig:bif3} shows the different bifurcation pathways that occur. 

\begin{figure}[h!]
\centering
\includegraphics[scale=0.25]{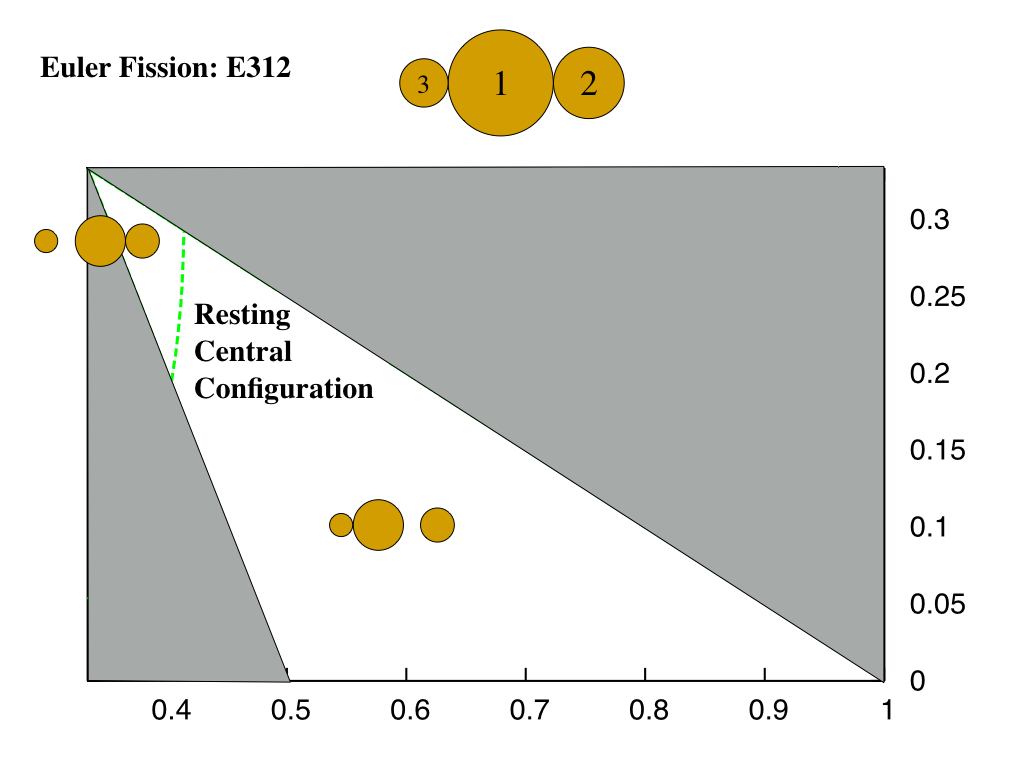}
\caption{Fission chart for the ER312 configuration. For masses to the right of the line the ER312 configuration fissions by having the intermediate mass separate and transitioning into the EA31-2 configuration. For masses to the left of the line the ER312 configuration fissions by having the smallest mass separate and transitioning into the EA21-3 configuration.}
\label{fig:EA312}
\end{figure}

\paragraph{Stability:}

For an EA$ij$-$k$ configuration to be stable requires $\delta_{d_{ij}}{\cal E} > 0$ and the second variations of ${\cal E}$ with respect to $d_{jk}$ and $\theta_{ki}$ be positive definite. The condition on $d_{ij}$ is automatically satisfied, except at specific transition points, once it is shown that a given configuration exists. Thus only the second order variation conditions need to be evaluated. 

First, note that the cross variations $\delta^2_{d_{jk}\theta_{ki}}{\cal E}$ are identically zero, and only consider the second variations $\delta^2_{\theta_{ki}}{\cal E}$ and $\delta^2_{d_{jk}}{\cal E}$ separately. Computing the first of these and evaluating it at the nominal condition yields
\beq
	\delta^2_{\theta_{ki}}{\cal E} & = & - m_k m_i \left[ - \frac{H^2}{I_H^2} + \frac{1}{d_{ki}^3}\right] (1-r_k) \left( d_{ki} - (1-r_k)\right) (\delta\theta_{ki})^2
\eeq

Make the substitution $H^2/I_H^2 = \frac{1}{d_{ki}^3} F( m_j/m_i, 1 - (1-r_k)/d_{ki})$. Then, stability in this variation can be established by showing that
\beq
	F(m_j/m_i, 1 - (1-r_k)/d_{ki}) & > & 1
\eeq
However, the function $F(\mu,1-1/r)$ was shown to be monotonically decreasing in $r$ with the limiting value of 1 as $r$ becomes arbitrarily large. Thus this is always satisfied and the EA configurations are always stable to variations in the angle $\theta_{ki}$. 

All that is left is to consider when $\delta^2_{d_{jk}}{\cal E} > 0$. First, re-express Eqn.\ \ref{eq:3bpjk} as
\beq
	\delta_{d_{jk}} {\cal E} & = & \frac{ m_k}{I_H^2}  \left\{ - H^2\left( m_j + m_i \right) +  I_H^2 \left[  \frac{m_j}{d_{jk}^2}  + m_i  \frac{d_{jk} - d_{ij}\cos\theta_{ki}}{d_{ki}^3} \right]  \right\} \delta d_{jk} 
\eeq
The term inside the brackets is identically zero at equilibrium, thus one does not need to take the variation of terms outside of the brackets. Taking the variation inside the brackets and simplifying yields
%\beq
%	\delta^2_{d_{jk}} {\cal E} & = & \frac{ m_k}{I_H^2} \left\{ 4 m_k I_H \left( m_j d_{jk} + m_i d_{ki} \right) \left[  \frac{m_j}{d_{jk}^2}  + m_i  \frac{1}{d_{ki}^2} \right] \right. \nonumber \\
%	& & \left. - 2 I_H^2 \left[ \frac{m_j}{d_{jk}^3} + \frac{m_i}{d_{ki}^3} \right] \right\}
%\eeq
\beq
	\delta^2_{d_{jk}} {\cal E} & = & 2 m_k \left\{ \frac{2 m_k}{ I_H } \left( m_j d_{jk} + m_i d_{ki} \right) \left[  \frac{m_j}{d_{jk}^2}  +   \frac{m_i}{d_{ki}^2} \right]  - \left[ \frac{m_j}{d_{jk}^3} + \frac{m_i}{d_{ki}^3} \right] \right\} \delta d_{jk} ^2
\eeq
First consider the case when $d \sim d_{jk} \sim d_{ki} \gg 1$. The second variation then reduces to
\beq
	\delta^2_{d_{jk}} {\cal E} & = & 2 m_k \frac{m_j+m_i}{d^3}  \delta d_{jk} ^2 
\eeq
which is always positive. Thus, all EA$ij$-$k$ configurations with large enough distances are stable. 

Conversely, Corrolary \ref{corrolary:X} shows that the EA configurations are always unstable when they approach and touch the ER configurations. From this, consider the stability of the EA configurations as a function of separation. At their minimum separation, when terminating the ER configurations, they begin as unstable. As the distance is increased they eventually become stable, indicating that ${\cal E}_{d_{jk}d_{jk}}$ evaluated at the EA configuration must cross through zero at some specific equilibrium configuration, indicating the point where the $H$-Bifurcation occurs. 

\paragraph{Bifurcation:}
Studying this aspect of the situation can provide qualitative insight into how the EA configurations bifurcate into existence and evolve as $H$ is increased. Consider Eqn.\ \ref{eq:H_EA}, which must be satisfied for an EA relative equilibrium configuration. As all of the terms on the right hand side are positive and bounded from below, there is an absolute minimum value such that if $H$ is below this value the equality cannot be satisfied and the EA relative configuration does not exist. At this value of angular momentum there will be a bifurcation from no relative equilibria to two relative equilibria, corresponding to the point identified above where ${\cal E}_{d_{jk}d_{jk}} = 0$. As $H$ increases further one branch of the EA relative equilibria must migrate towards the ER configuration and the other to larger distances, due to the uniqueness of this family. The branch that migrates to the larger separation will be stable while the branch that migrates to the contact configuration must be unstable, from Lemma \cite{theorem:X}. A similar bifurcation will occur for the configuration that intersects with the Euler Orbital family. 

\subsection{No Active Constraint}

Finally consider the case where none of the constraints are active. Then the three conditions must all be identically zero. Due to the structure of the problem, it is well known that there are only 5 relative equilibria to this problem \cite{wintner}. These are divided into the Lagrange solutions, which lie at the vertices of an equilateral triangle \cite{lagrange}, and the Euler solutions, which lie in a single line and are appropriately spaced \cite{euler}. 

\subsubsection{Lagrange Solutions}

For the Lagrange solutions, set $d = d_{12} = d_{23} = d_{31}$. From Eqn.\ \ref{eq:d31} and note that this requires $\cos\theta_{31} = 1/2$, meaning that $\theta_{31} = \pm 60^\circ$. Then the condition can be uniformly satisfied by choosing the distance $d$ such that $H^2 = \frac{I_H(d)^2}{d^3}$. Note that for all $d> \max(r_i+r_j) = r_1+r_2$ (given the assumed ordering) such a solution will always exist.  However, for a given $H^2$ a solution to the non-contact case may not always exist. Indeed, since for this case $I_H = (m_1m_2 + m_2m_3 + m_3m_1) d^2 + I_S$, $H$ has a minimum value that is greater than zero, and thus will not exist for all values of angular momentum. 

The point where the Lagrange Orbital configuration comes into existence can be explicitly probed. In general there are two possibilities. One is that it appears as a two branch family once the angular momentum goes above its minimum value. Then one branch will migrate inwards with increasing angular momentum and terminate by intersection with the IS12-3 family. Otherwise, if the minimum angular momentum point arises at a mutual distance less than $r_1+r_2$, then the inner IS12-3 family will transition directly into the LO family, as it is known that the equal mass case has this sort of a bifurcation \cite{scheeres_minE}, it is relevant to test for when this will occur. To do this just compute $\partial H/\partial d$ and solve for the zero to find 
\beq
	d^2 & = & \frac{3 I_S}{m_1 m_2 + m_2 m_3 + m_3 m_1}
\eeq
Which bifurcation structure ensues can be found by finding where this solution is greater or less than $(1-r_3)$. Figure \ref{fig:LOH} plots the region where the double branch occurs and where the single branch occurs. Figure \ref{fig:bif4} shows the two different bifurcation pathways. 

\begin{figure}[h!]
\centering
\includegraphics[scale=0.25]{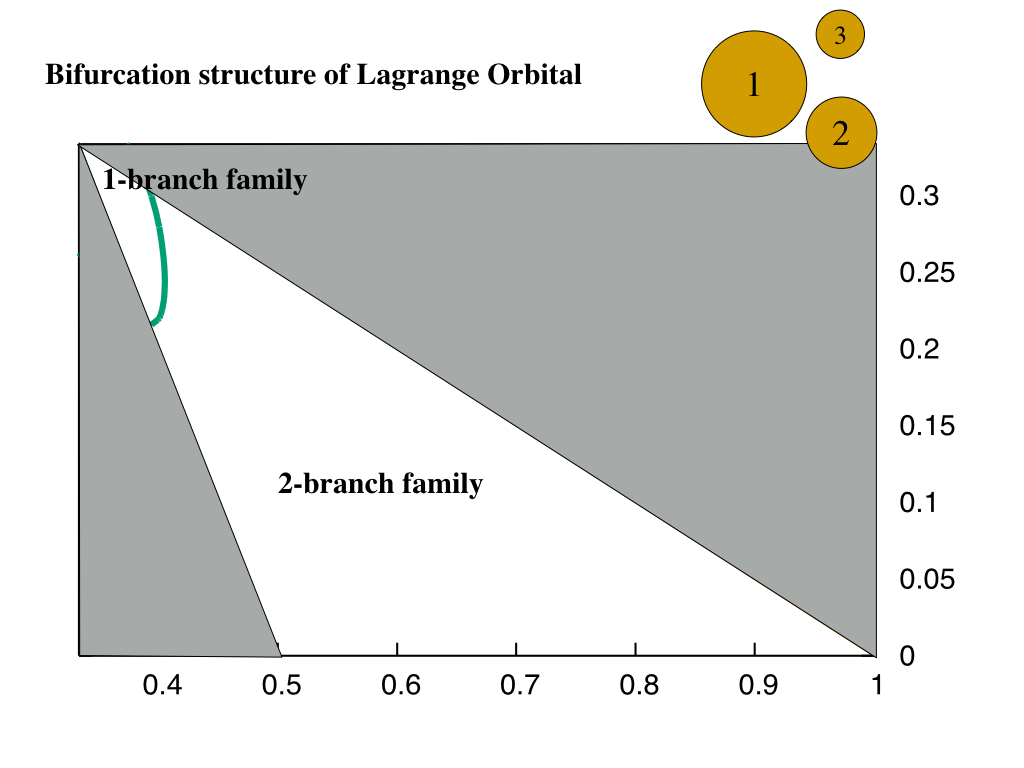}
\caption{Transition line between the two different bifurcation modes of the Lagrange Orbital configuration. The two different patterns of bifurcation are shown in Fig.\ \ref{fig:bif4}. }
\label{fig:LOH}
\end{figure}

While it is well known that the classical 3-body problem is spectrally stable when the Routh Criterion is satisfied, it should be noted that the stability considered in this paper, energetic stability, is a stronger type of stability. An observation by Moeckel \cite{moeckel_central} shows that central configurations in the point-mass $N$-body problem never have a positive definite second variation of their energy, and thus it can be suspected that the same holds true for the Lagrange Orbital configuration in the Full body. 
To test this, take the second order variation of $\mathcal{E}$, evaluated at the equilibrium, and determine if the resulting matrix of values is positive definite. Here it is simpler to take the 3 distances $d_{12}$, $d_{23}$ and $d_{31}$ as the degrees of freedom, with the general form, starting from Eqn.\ \ref{eq:dsym}, substituting the equilibrium condition and simplifying
\beq
	\frac{\partial^2\mathcal{E}}{\partial d_{ij}^2} & = & \frac{m_i m_j}{d I_H} \left[ \left(m_i m_j - 3 m_k(m_i+m_j)\right) d^2 - 3 I_S\right] \\
	\frac{\partial^2\mathcal{E}}{\partial d_{ij}\partial d_{jk}} & = & \frac{4 m_i m_j^2 m_k}{d I_H} 	
\eeq
where $d \ge 1-r_3$ for the specific case of interest. 

For the full Hessian of $\mathcal{E}$, $[\partial^2\mathcal{E}/\partial d_{ij} d_{jk} ]$, to be positive definite utilize Sylvester's Theorem again, which states that a necessary and sufficient condition is that all of the principal minors of the Hessian matrix be positive. Thus, a necessary condition for being positive definite is that the diagonals all be positive. Should any of these be negative, then the matrix is not positive definite and hence the relative equilibrium configuration is not energetically stable. Consider the entry for $\frac{\partial^2\mathcal{E}}{\partial d_{23}^2}$. The controlling condition for stability is then that $\left[ m_2 m_3 - 3 m_1(m_2+m_3)\right] d^2 - 3 I_S$ be positive. 
However, it is easy to show that the term $\left[ m_2 m_3 - 3 m_1(m_2+m_3)\right] < 0$, showing that the Lagrange Orbital configurations are always energetically unstable, consistent with Moeckel's result. First, restate the negative condition as $3 m_1 (m_2 + m_3) > m_2 m_3$, then note that $m_2 + m_3 > m_3$ and $m_1 > m_2$, establishing the inequality unequivocally. 

\subsubsection{Euler Solutions}

For the Euler conditions consider $d_{ij} \ge r_i+r_j$, $d_{jk} \ge r_j+r_k$, and $\theta_{ki} = \pi$. This case also has $d_{ki} = d_{ij} + d_{jk}$. Both Eqns.\ \ref{eq:3bpij} and \ref{eq:3bpjk} must equal zero in this case, yielding the two conditions.
\beq
	0 & = & m_i m_j \left[ - \frac{H^2}{I_H^2} + \frac{1}{d_{ij}^3} \right] d_{ij} + m_i m_k \left[ - \frac{H^2}{I_H^2} + \frac{1}{d_{ki}^3} \right] d_{ki}   \\
	0 & = & m_k m_j \left[ - \frac{H^2}{I_H^2} + \frac{1}{d_{jk}^3} \right] d_{jk} + m_k m_i \left[ - \frac{H^2}{I_H^2} + \frac{1}{d_{ki}^3} \right]  d_{ki} 
\eeq
First, there is a more fundamental equality within these results
\beq
	m_i m_j \left[ - \frac{H^2}{I_H^2} + \frac{1}{d_{ij}^3} \right] d_{ij} & = m_k m_i \left[ \frac{H^2}{I_H^2} - \frac{1}{d_{ki}^3} \right]  d_{ki} = & m_k m_j \left[ - \frac{H^2}{I_H^2} + \frac{1}{d_{jk}^3} \right] d_{jk} 
\eeq
By inspection, with the knowledge that $d_{ki} \ge d_{jk}, d_{ij}$, note that 
\beq
	\frac{1}{d_{ki}^3} & \le \frac{H^2}{I_H^2} \le & \frac{1}{\max{ (d_{jk}, d_{ij})}^3}
\eeq
Alternately, this ratio can also be solved for the quantity $(H/I_H)^2$ to find
\beq
	\frac{H^2}{I_H^2}  & = & \frac{1}{m_j d_{ij} + m_k d_{ki}} \left[ \frac{m_j}{d_{ij}^2} + \frac{m_k}{d_{ki}^2} \right] \\
	\frac{H^2}{I_H^2}  & = & \frac{1}{m_j d_{jk} + m_i d_{ki}} \left[ \frac{m_j}{d_{jk}^2} + \frac{m_i}{d_{ki}^2} \right] 
\eeq
which is the condition used to analyze how the ER configurations fissioned. Indeed, at the transition lines on Figs.\ \ref{fig:EA123} and \ref{fig:EA312} the resting configuration is in fact a central configuration, meaning that the relative attractions between the bodies will be balanced so long as their relative distances are preserved. 

This can be generalized to identify the possible bifurcation pattern in the EO configurations. Assume, say, that bodies $i$ and $j$ are in contact and that as body $k$ is moved to a larger distance it reaches the point where the equality between the above conditions occurs, meaning that bodies $i$ and $j$ are now in a relative equilibrium condition and the entire system satisfies a central configuration conditions \cite{wintner}. At this point the relative distances between these bodies can be uniformly scaled with the ratio $H/I_H$ following along. The change in angular momentum with this scaling is not uniform, however, due to the $3I_S$ term in $I_H$. Of specific interest regarding the pattern of bifurcation is whether the angular momentum decreases or increases with this change in relative distance. In the following it can be shown that both conditions can occur in general. 

Define the distance between bodies $k$ and $i$ where the EO conditions are satisfied (assuming $d_{ij} = 1-r_k$) as $d_{ki}^*$, and thus $d_{jk}^* = d_{ki}^* - (1-r_k)$, and define the ratio $H/I_H = \Omega^*$ at this point. Then, for increasing the distance the relative equilibria will all scale uniformly, meaning that 
\beq
	\frac{H^2}{I_H^2} & = & \frac{\Omega^{*2}}{d^3}
\eeq
where $d \ge 1$ and $d_{ij} = d(1-r_k)$, $d_{jk} = d d^*_{jk}$ and $d_{ki} = d d^*_{ki}$. With this structure, it is possible to compute the gradient 
\beq
	\frac{\partial H}{\partial d} & = & \frac{\Omega^{*2}}{2 d^{5/2}} \left[ \left( m_i m_j (1-r_k)^2 + m_j m_k (d_{ki}^{*}-(1-r_k))^2 + m_k m_i d_{ki}^{*2}\right) d^2 - 3 I_S\right] 
\eeq
How the bifurcation occurs can be tested by plotting the level sets from $d_{ki}^* = (1+r_j)$ to large values. For the 1, 2, 3 and 2, 1, 3 orderings the gradient is positive towards the apex of the triangular region and can take on negative values near the base. Thus, the appearance of the EO orbits occur as a transition closer to the equal mass condition and as a bifurcation followed by a termination away from there. Precise limits could be computed, but would require root solving algorithms.

Finally, consider the stability of the Euler Orbital solutions. These are again suspected to be energetically unstable due to the instability of the point mass cases, however this should be checked given the changes in the current approach. First, note that the second order variation in $\theta_{ki}$ is uncoupled from the variations in distance, and evaluated at the equilibrium yields
\beq
	\delta^2_{\theta_{ki}} {\cal E} & = & - m_k m_i \left[ - \frac{H^2}{I_H^2} + \frac{1}{d_{ki}^3} \right] d_{ij} d_{jk} (\delta \theta_{ki})^2  
\eeq
The quantity in the brackets is negative, as established above, and hence the angle variation is stable. For the distance variations the full $2\times 2$ Hessian matrix must be evaluated, however one can again just check the necessary conditions that the diagonals must all be positive. Taking the second order variations of both conditions from the diagonal of the Hessian matrix shows that both of the following conditions must be positive for stability
\beq
	m_i m_j - 3 (m_i+m_j)m_k & > & 0 \\
	m_j m_k - 3 (m_j+m_k)m_i & > & 0 
\eeq
For the ordering $m_3 \le m_2 \le m_1$ it can be shown that for all combinations of $i$, $j$ and $k$ that at least one of these conditions will be violated, and hence the EO configurations are always unstable. 

\end{proof}

\section{Summary}

To finish, the results are presented in light of the main theorem. First, consider the total number of relative equilibria found. For the no contact case recall the classical result of 5 distinct relative equilibria. When one contact is active the EA and IS relative equilibria were identified, which have 6 unique components each, raising the count to 17. For the two contact cases there are 3 ER and 6 TR configurations, resulting in a total of 26. Finally, for three contact cases there are the 2 LR configurations, leading to the total of 28. 

Now consider the bifurcation patterns, which are focused on the transitions between the different contact cases, and the identification of when the $H$-Bifurcations can occur. The details of the transitions will be outlined, although a few observations can be given first. With regard to stability the system starts with only two stable LR configurations at low values of $H$ and eventually has six stable EA configurations for arbitrarily large values of $H$. Between these limits the number of stable configurations can vary, and to establish the precise sequence and number would require a more detailed investigation for a specific set of sizes. It is noted, however, that there always exist at least one stable configuration. 

The bifurcation pattern seen in Fig.\ \ref{fig:bif2} is described first. 
The LR configurations all end at a Termination Fission condition with the TR132 configurations. These configurations either arise from an $H$-Bifurcation (in a limited region of the parameter space) or more commonly emerge as a symmetric bifurcation as the unstable ER132 configuration stabilizes. Under increasing $H$ the ER132 configuration either ends with a Termination Fission with the unstable component of the EA13-2 configuration, or for a limited range of parameters ends with a Transition Fission into the EA13-2 configuration. This second occurrence is of great interest as it is the only occasion in which the Spherical Full 3-Body Problem will fission into a stable configuration. The EA13-2 configuration itself usually arises as an $H$-Bifurcation, with its unstable branch terminating as mentioned above and its stable branch existing for all $H$ with an increasing distance proportional to $H^2 $. The only exception is when it arises as a Transition Fission, as described above. The EA23-1 configuration has an evolution that is completely isolated from the rest of this chart. It arises through an $H$-Bifurcation, with its unstable component either having a Termination Fission with the EO132 configuration or a Transition Fission into an EO132 configuration. Its stable component continues for all larger values of $H$ with a similar asymptotic form as the EA13-2 configuration. The EO132 configuration can either arise as an $H$-Bifurcation or through a Transition Fission, however the EO components are always unstable. 
It is important to note that at each stage of the system evolution with $H$, that there is at least one stable configuration, providing proof of that aspect of the theorem. 

Now consider the bifurcation patterns for the ER123 and ER213 pathways, shown in Figs.\ \ref{fig:bif1} and \ref{fig:bif3}. These are similar, and distinct from the ER132 pathway. Each of these start out in unstable ER configurations and both stabilize by a symmetric bifurcation with the TR123 and TR213 configurations, respectively. The TR123 and TR213 configurations end with a Termination Fission with an IS23-1 and IS13-2 configuration, respectively. The IS configurations arise through an $H$-Bifurcation with the inner configuration ending with the Termination Fission mentioned above and the outer configuration extending for all $H$, ultimately with their size on the order of $H^2$, although the IS configurations are always unstable. The stable ER123 and ER213 configurations end with a Termination Fission into an EA12-3 or EA32-1 configuration for the ER123 case or an EA21-3 or EA31-2 configuration for the ER213 case. Limits where these transitions occur have been delineated in Figs.\ \ref{fig:EA123} and \ref{fig:EA312}. Note that at the transition between these fission pattern the ER123 and ER213 configurations are in a central configuration, a situation that does not happen for the ER132 configuration. When the EA inner configurations do not end with a Termination Fission with an ER configuration, they end with a Termination Fission or Transition Fission with the corresponding EO configuration. The outer EA configurations are all stable and have the same asymptotic structure for large $H$. The EO123 and EO213 configurations either arise as an $H$-Bifurcation or as a Transition Fission, with the details of these boundaries left for future investigation. 

Finally consider the sequence involving the LO configuration, represented in Fig.\ \ref{fig:bif4}. This sequence is the least complex, with the IS12-3 configurations arising from an $H$-Bifurcation. The inner component ends with a Termination Fission or Transition Fission with the LO configuration. The LO configuration, in turn, either arises as an $H$-Bifurcation, with the inner component ending with a Termination Fission, or as a Transition Fission. All configurations in these sequences are unstable.

\section*{Appendix}

%\begin{lemma}
%\label{lemma:X}
%Assume a resting configuration exists in 1-D, i.e., $\left. \delta_{d}{\cal E}\right|_{d=0} > 0$, and is defined such that $\delta d \ge 0$. Then, if the resting condition is intersected by a relative equilibrium condition defined at a distance $d^* \rightarrow 0^+$, i.e., $\left. \delta_{d}{\cal E}\right|_{d^*} = 0$, then $\left. \delta^2_{d}{\cal E}\right|_{d^*} < 0$ and the relative equilibria at $d^*$ is Energetically Unstable. 
%\end{lemma}
%
%\begin{proof}
%Expand the minimum energy function about the equilibrium point $0 < d^* \ll 1$ to find ${\cal E}_d(d^* + \Delta d) = {\cal E}^*_{dd} \Delta d + \ldots$, as the first derivative is zero. Then substituting $\Delta d = - d^*$ and evaluating the first variation of energy we find $\left. \delta_d{\cal E}\right|_{d=0} = - {\cal E}^*_{dd} d^* \delta d > 0$. But for an allowable variation we have $\delta d \ge 0$, and thus ${\cal E}^*_{dd} < 0$ and the equilibrium point at $d^*$ is Energetically Unstable by Lemma \ref{lemma:3}. 
%\end{proof}

\subsection*{Partial Derivatives}

It is useful to state the relevant partial derivatives of the amended potential and its constituent terms, as a function of the distances and angles. In the following use the convention that the distances are denoted with indices $ij$ and $jk$ and the angle with indices $ki$. 

If the third degree of freedom is the angle $\theta_{ki}$ then 
\beq
	\frac{\partial {\cal E}}{\partial d_{ij}} & = & - \frac{H^2}{2 I_H^2} \frac{\partial I_H}{\partial d_{ij}} + \frac{\partial \mathcal{U} }{\partial d_{ij}} \\
%	\frac{\partial {\cal E}}{\partial d_{jk}} & = & - \frac{H^2}{2 I_H^2} \frac{\partial I_H}{\partial d_{jk}} + \frac{\partial \mathcal{U} }{\partial d_{jk}} \\
	\frac{\partial {\cal E}}{\partial \theta_{ki}} & = & - \frac{H^2}{2 I_H^2} \frac{\partial I_H}{\partial \theta_{ki}} + \frac{\partial \mathcal{U} }{\partial \theta_{ki}} 
\eeq
where
\beq
	\frac{\partial I_H}{\partial d_{ij}} & = & 2 m_i m_j d_{ij} + 2 m_i m_k \left( d_{ij} - d_{jk} \cos\theta_{ki} \right) \\
%	\frac{\partial I_H}{\partial d_{jk}} & = & 2 m_j m_k d_{jk} + 2 m_i m_k \left( d_{jk} - d_{ij} \cos\theta_{ki} \right) \\
	\frac{\partial I_H}{\partial \theta_{ki}} & = & 2 m_i m_k d_{ij} d_{jk} \sin\theta_{ki} 
\eeq
Similarly
\beq
	\frac{\partial \mathcal{U} }{\partial d_{ij}} & = & m_i m_j \frac{1}{d_{ij}^2} + m_i m_k \frac{\left( d_{ij} - d_{jk} \cos\theta_{ki} \right)}{d_{ki}^3} \\
%	\frac{\partial \mathcal{U} }{\partial d_{ij}} & = & m_j m_k \frac{1}{d_{jk}^2} + m_i m_k \frac{\left( d_{jk} - d_{ij} \cos\theta_{ki} \right)}{d_{ki}^3} \\
	\frac{\partial \mathcal{U} }{\partial \theta_{ki}} & = & m_i m_k \frac{ d_{ij} d_{jk} \sin\theta_{ki} }{d_{ij}^3}
\eeq

If the third degree of freedom is the distance $d_{ki}$, and if not at a limiting constraint, then 
\beq
	\frac{\partial {\cal E}}{\partial d_{ij}} & = & - \frac{H^2}{2 I_H^2} \frac{\partial I_H}{\partial d_{ij}} + \frac{\partial \mathcal{U} }{\partial d_{ij}} 
\eeq
where
\beq
	\frac{\partial I_H}{\partial d_{ij}} & = & 2 m_i m_j d_{ij} 
\eeq
and
\beq
	\frac{\partial \mathcal{U} }{\partial d_{ij}} & = & m_i m_j \frac{1}{d_{ij}^2} 
\eeq

\subsection*{Equilibrium Conditions}

For a relative equilibrium there are two different possibilities. Either $\delta{\cal E}=0$ or $\delta{\cal E} > 0$. 
For either, the relevant statement of the variations is given in the following for the two different formulations of the third degree of freedom.

If the third degree of freedom is the angle $\theta_{ki}$ then the full set of variations are
\beq
	\delta_{d_{ij}} {\cal E} & = & m_i \left\{ m_j \left[ - \frac{H^2}{I_H^2} + \frac{1}{d_{ij}^3} \right] d_{ij} \right. \nonumber \\
	& & \left. + m_k \left[ - \frac{H^2}{I_H^2} + \frac{1}{d_{ki}^3} \right] \left[ d_{ij} - d_{jk}\cos\theta_{ki}\right] \right\} \delta d_{ij} \label{eq:3bpij} \\
	\delta_{d_{jk}} {\cal E} & = & m_k \left\{ m_j \left[ - \frac{H^2}{I_H^2} + \frac{1}{d_{jk}^3} \right] d_{jk} \right. \nonumber \\
	& & + \left. m_i \left[ - \frac{H^2}{I_H^2} + \frac{1}{d_{ki}^3} \right] \left[ d_{jk} - d_{ij}\cos\theta_{ki}\right] \right\} \delta d_{jk} \label{eq:3bpjk}
\eeq
%\beq
%	\delta_{d_{23}} {\cal E} & = & m_2 m_3 d_{23} \left[ - \frac{H^2}{I_H^2} + \frac{1}{d_{23}^3} \right] \delta{d_{23}} \nonumber \\
%	& & + m_3 m_1 \left[ - \frac{H^2}{I_H^2} + \frac{1}{d_{31}^3} \right] \left[ d_{23} - d_{12}\cos\theta_{31}\right] \delta d_{23}  \label{eq:3bp2}
%\eeq
\beq
	\delta_{\theta_{ki}} {\cal E} & = & m_k m_i \left[ - \frac{H^2}{I_H^2} + \frac{1}{d_{ki}^3} \right] d_{ij} d_{jk} \sin\theta_{ki} \delta \theta_{ki}  \label{eq:3bpki}
\eeq

If the third degree of freedom is the distance $d_{ki}$ and the system is not at a constraint limit (i.e., $d_{ki} \ne |d_{ij} \pm d_{jk}|$), then the full set of variations are
\beq
	\delta_{d_{ij}} {\cal E} & = & m_i m_j \left[ - \frac{H^2}{I_H^2} + \frac{1}{d_{ij}^3} \right] d_{ij} \delta d_{ij}  \label{eq:dsym} \\
	\delta_{d_{jk}} {\cal E} & = & m_j m_k \left[ - \frac{H^2}{I_H^2} + \frac{1}{d_{jk}^3} \right] d_{jk} \delta d_{jk}  \\
	\delta_{d_{ki}} {\cal E} & = & m_k m_i \left[ - \frac{H^2}{I_H^2} + \frac{1}{d_{ki}^3} \right] d_{ki} \delta d_{ki}  
\eeq

These expressions are used to develop the necessary and sufficient conditions for a configuration to be a relative equilibrium. If any two components are in contact, then the condition for the degree of freedom that is blocked should be $\delta{\cal E} > 0$ for all allowed variations, otherwise the condition should be $\delta{\cal E} = 0$. 

\subsection*{Properties of the Function $F(\mu,x)$}
\begin{lemma}
The function
\beq
	F(\mu,x) = \frac{1+\frac{\mu}{x^2}}{1+\mu x}
\eeq
is monotonically decreasing in $x$ and is convex in $x$ over the interval $x\in ( 0, \infty)$.
\end{lemma}

\begin{proof}
Consider the first derivative of the function with respect to $x$:
\beq
	F'(\mu,x) & = & \frac{- 2 \mu}{x^3} \frac{1}{1+\mu x} - \mu \frac{1+\frac{\mu}{x^2}}{(1+\mu x)^2}
\eeq
By inspection it can be seen that all terms are negative and non-zero, and thus the function is monotonically decreasing in $x$.

Taking the second derivative with respect to $x$:
\beq
	F''(\mu,x) & = & \frac{6 \mu}{x^4} \frac{1}{1+\mu x} + \frac{4 \mu^2}{x^3} \frac{1}{(1+\mu x)^2} 
	+ 2 \mu^2 \frac{1+\frac{\mu}{x^2}}{(1+\mu x)^3}
\eeq
By inspection again see that all terms are positive and non-zero and is thus convex. 
\end{proof}

\section*{Acknowledgements}

The author acknowledges support from NASA grant NNX14AL16G from the Near Earth Objects Observation programs. 

\newpage \bibliographystyle{plain}
\bibliography{../../../bibliographies/biblio_article,../../../bibliographies/biblio_books,../../../bibliographies/biblio_misc}

\end{document}